\documentclass{article}
\usepackage{graphicx} %
\usepackage{graphicx} %
\usepackage{amsmath,amssymb,amsthm}
\usepackage{algpseudocode}
\usepackage{algorithm}
\usepackage{fullpage}
\usepackage[bookmarks]{hyperref}
\usepackage{xcolor}
\usepackage{mleftright}
\usepackage[shortlabels]{enumitem}
\usepackage{bookmark}
\bookmarksetup{numbered}

\newcommand{\qand}{\quad\And\quad}

\newcommand{\B}{{\mathcal{B}}}
\newcommand{\net}{{\mathcal{N}}}
\newcommand{\C}{{\mathbb{C}}}
\newcommand{\R}{{\mathbb{R}}}

\renewcommand{\H}{{\mathcal{H}}}

\newcommand{\eps}{\varepsilon}

\renewcommand{\Re}{\mathrm{Re}}
\renewcommand{\Im}{\mathrm{Im}}

\newcommand{\wrt}{\,\textnormal d}
\newcommand{\deriv}[2]{\frac{\wrt^{#2}}{\wrt{#1}^{#2}}}

\newcommand{\wt}{\widetilde}
\newcommand{\conj}{\overline}

\newcommand{\abs}[1]{\mleft|#1\mright|}
\newcommand{\rabs}[1]{\mleft|#1\mright|}
\newcommand{\magn}[1]{\left\|#1\right\|}

\newcommand{\pare}[1]{\mleft(#1\mright)}

\newcommand{\sqbrac}[1]{{\left[{#1}\right]}}
\newcommand{\set}[1]{{\left\{{#1}\right\}}}

\newcommand{\bmat}[1]{\begin{bmatrix}#1\end{bmatrix}}

\newcommand{\spliteq}[2]{\begin{equation}#1\begin{split}#2\end{split}\end{equation}}

\DeclareMathOperator*{\E}{\mathbb{E}}

\DeclareMathOperator*{\argmin}{arg\,min}

\DeclareMathOperator*{\poly}{poly}
\DeclareMathOperator*{\col}{Col}

\DeclareMathOperator*{\spn}{span}
\DeclareMathOperator*{\rank}{rank}

\DeclareMathOperator*{\area}{Area}
\DeclareMathOperator*{\conv}{conv}

\DeclareMathOperator*{\range}{range}
\DeclareMathOperator*{\ball}{Ball}
\DeclareMathOperator*{\inR}{inr}

\DeclareMathOperator*{\len}{len}

\newenvironment{rcases}
  {\left.\begin{aligned}}
  {\end{aligned}\right\rbrace}
\newenvironment{lcases}
  {\left\lbrace\begin{aligned}}
  {\end{aligned}\right.}

\newcommand{\smin}{\sigma_{\min}}

\newtheorem{theorem}{Theorem}[section]
\newtheorem{lemma}[theorem]{Lemma}
\newtheorem{proposition}[theorem]{Proposition}
\newtheorem{corollary}[theorem]{Corollary}

\newtheorem{fact}[theorem]{Fact}

\theoremstyle{definition}
\newtheorem{definition}{Definition}[section]

\algdef{SE}[REPEATN]{RepeatN}{End}[1]{\algorithmicrepeat\ #1 \textbf{times}}{\algorithmicend}

\makeatletter
\newtheorem*{rep@theorem}{\rep@title}
\newcommand{\newreptheorem}[2]{%
\newenvironment{rep#1}[1]{%
 \def\rep@title{#2 \ref{##1}}%
 \begin{rep@theorem}}%
 {\end{rep@theorem}}}
\makeatother

\newreptheorem{theorem}{Theorem}
\newreptheorem{coro}{Corollary}

\title{The Pseudospectrum of Random Compressions of Matrices}
\author{Rikhav Shah\footnote{Supported by NSF CCF-2420130}\\UC Berkeley}
\date{\today}

\begin{document}
\maketitle

\newcommand{\semiu}{\textnormal{\~U}}
\newcommand{\constepso}{c_{1,\ell,n}}
\newcommand{\constepst}{c_{2,\ell,n}}
\newcommand{\constlog}{c_{3,\ell,n}}

\begin{abstract}
    The compression of a matrix
    \(A\in\mathbb C^{n\times n}\)
    onto a subspace \(V\subset\mathbb C^n\)
    is the matrix $Q^*AQ$ where the columns of $Q$ form an orthonormal basis for $V$. This is an important object in both operator theory and numerical linear algebra. Of particular interest are the eigenvalues of the compression and their stability under perturbations.
    This paper considers compressions onto subspaces sampled from the Haar measure on the complex Grassmannian.
    We show the expected area of the \(\varepsilon\)-pseudospectrum of such compressions is bounded by \(\textnormal{poly}(n)\log^2(1/\varepsilon)\cdot\eps^\beta\), where \(\beta=6/5,4/3\), or \(2\) depending on some mild assumptions on $A$.
    Along the way, we obtain (a) tail bounds for the least singular value of compressions and (b) non-asymptotic small-ball estimates for random non-Hermitian quadratic forms surpassing bounds achieved by existing methods.
\end{abstract}

\section{Introduction}\label{a0}
The pseudospectrum of a matrix is defined as the set of all eigenvalues of all nearby matrices,
\spliteq{\label{a1}}{\Lambda_\eps(M)=\set{\lambda\in\C:\lambda\text{ is an eigenvalue of }M+E\text{ for some }\magn E\le\eps}.}
The area of this set can be interpreted as a measure of the stability of the eigenvalues of $M$. When $M$ is a normal matrix, the eigenvalues are as stable as one could hope for: they are $1$-Lipschitz functions of the matrix with respect to the operator norm.
Thus, the set $\Lambda_\eps(M)$ is simply the union of disks of radius $\eps$ around each eigenvalue. At the other extreme, when $M$ is a single Jordan block, the eigenvalues are as instable as possible, and the set $\Lambda_\eps(M)$ is a disk of radius $\eps^{1/n}$ where $n$ is the matrix dimension \cite{b0}.

Given a complex matrix $A\in\C^{n\times n}$, this paper estimates the expected area of the $\eps$-pseudospectrum of a random compression of $A$, i.e.
\[\E\area\Lambda_\eps(Q^*AQ)\]
where $Q\in\C^{n\times\ell}$ is an orthonormal basis for a random subspace sampled from the Haar distribution on the complex Grassmannian $\text{Gr}_\ell(\C^n)$.
There are deterministic lower and upper bounds on $\area\Lambda_\eps(Q^*AQ)$ coming from the normal matrix case and Jordan block case respectively,
\[\pi\eps^2\le\area\Lambda_\eps(Q^*AQ)\le\pi\eps^{2/\ell}.\]
This work shows that for small $\eps$, under mild assumptions on $A$, the expected value of $\area\Lambda_\eps(Q^*AQ)$ is much closer to the deterministic lower bound than to the deterministic upper bound.
The key assumptions on $A$ are the following, where $W(A)$ is the numerical range of $A$. 
\begin{enumerate}[(a)]
    \item $W(A)$ is contained in a disk of radius $\poly(n)$,
    \item $W(A)$ contains a disk of radius $1/\poly(n)$,
    \item $\inf_{z\in\C}\sigma_{\ell+8}(z-A)\ge1/\poly(n)$.
\end{enumerate}
Assumption (c) is motivated by the observation that if $\sigma_\ell(z-A)=0$, then $z\in\Lambda_\eps(Q^*AQ)$ with probability 1 since $\rank(z-Q^*AQ)\le\rank(z-A)<\ell$. In other words, it is difficult bound the probability a given point $z$ is in the pseudospectrum when $z-A$ is low-rank.

Our main theorem, Theorem \ref{a2}, provides several non-asymptotic bounds on $\E\area\Lambda_\eps(Q^*AQ)$ in terms of the quantities of $A$ associated with assumptions (a), (b), and (c). Using those assumptions results in the following corollary.%
\begin{corollary}\label{a3}
Let $Q\in\C^{n\times\ell}$ be an orthonormal basis for an $\ell\le n/2-8$ dimensional subspace sampled from the Haar distribution on $\textnormal{Gr}_\ell(\C^n)$. Set
\[
\beta=\begin{cases}
    6/5 & \text{if assuming (a)}\\
    4/3 & \text{if assuming (a) and (b)}\\
    2 & \text{if assuming (a) and (c)}
\end{cases}.
\]
Then
\spliteq{\label{a4}}{
\E\area(\Lambda_\eps(Q^*AQ))\le\poly(n)\log^2(1/\eps)\cdot\eps^\beta.}
\end{corollary}
Importantly, \eqref{a4} is close to achieving the optimal dependence on $\eps$ under only assumptions (a) and (c).

\subsection{Motivation}
The stability of eigenvalues of a random matrix under perturbations to its entries has been of recent interest \cite{b1,b2,b3,b4,b5,b6}.
Specifically, one seeks to understand which random matrices typically have eigenvalues that are stable.
Several existing works are motivated by the desire to rigorously analyze algorithms for eigenvalue problems, particularly in finite precision.
This work is similarly motivated by studying a popular class of eigenvalue algorithms, specifically Rayleigh-Ritz methods \cite{b7,b8}. These methods construct an orthonormal basis $Q\in\C^{n\times\ell}$ for some (specially chosen) subspace, and compute the eigenvalues of the compression $Q^*AQ\in\C^{\ell\times\ell}$. In many settings, the eigenvalues of the compression approximate some of the eigenvalues of the original matrix $A$.

The chosen subspace that $Q$ is a basis of depends on the particular Rayleigh-Ritz method used. Typically, one uses a randomized approximation of an invariant subspace of $A$, which can ultimately be a complicated distribution on $\text{Gr}_\ell(\C^n)$.
This paper considers the much simpler Haar distribution. This does not make for an {accurate} Rayleigh-Ritz method, but the distribution is natural and the results of this paper could serve as a starting point for understanding the {stability} of Rayleigh-Ritz methods more broadly.

\subsection{Related work}

\paragraph{Stucture-preserving compressions.}
In the Hermitian case, one readily observes that the compression of a Hermitian matrix remains Hermitian, so control over the size of the pseudospectrum is automatic. However, the compression of a normal matrix is typically not also normal. For example, a compression of a cyclic permutation matrix can be a single Jordan block. For $\dim V=n-1$, there's an explicit characterization of the cases when both a matrix and its compression are normal: it only occurs when $V$ is an invariant subspace or the spectrum of $A$ lies on a line (i.e. $A=z_1I+z_2H$ for some Hermitian $H$ and scalars $z_1,z_2\in\C$)
\cite{b9}. In particular, we can not expect this from random subspaces in general. Furthermore, we seek control of the pseudospectrum of the compression even when the original matrix $A$ is highly non-normal.

\paragraph{Smallest singular values.}
The standard way to control the expected area of the pseudospectrum is to pass to a lower tail bound on the smallest singular value of shifts
\cite{b1,b3,b4}. This comes from an equivalent definition of the pseudospectrum  as \eqref{a1},
\[\Lambda_\eps(M)
=\set{z\in\C:\smin(z-M)\le\eps}
.\]
Note that the entries of $z-Q^*AQ$ are highly dependent random variables.
There has been much work on producing lower tail bounds on $\smin(M)$ for random $M$, but the bulk of it strongly relies on working in the setting where $M-\E(M)$ has i.i.d. entries, e.g. \cite{b10,b11,b12,b13,b6}.

By the Schur formula for the inverse, one has the relation
\[Q(Q^*AQ)^{-1}Q^*=\lim_{t\to\infty}(A+t(I-QQ^*))^{-1}\]
Note that $I-QQ^*$ is a low-rank matrix. The work \cite{b14} studied the norm of $(A+UV^*)$ where $U$ and $V$ are independent Gaussian matrices. The problem in this paper can be roughly be seen as setting $U=V$ and taking the limit of their variance to infinity. This dependence between $U$ and $V$ in this setting makes it non-trivial to adapt the techniques of \cite{b14}.

\paragraph{Small ball probabilities for quadratic forms.}
An important quantity we bound is the small-ball probability for random quadratic forms \(\sup_{z\in\C}\Pr(|x^*Ax-z|\le\eps)\) where $x$ is sampled uniformly from the complex unit sphere.
The distribution of $x^*Ax$ is called the \textit{numerical measure} of $A$, and appears to have been first considered in detail by \cite{b15}. They provide an exact characterization of the distribution when $A$ is Hermitian, and give an integral expression involving polynomial splines for the density in general (\cite[Proposition 3.1]{b15}). Our work is built on top of this result. Their work also computes several interesting properties of these measures, such as the location of possible singularities and regions where most of the mass accumulates. However, they do not provide small-ball estimates.

A related distribution is that of $x^*Ax$ where $x$ has i.i.d. (complex) Gaussian entries. This has has been studied by many works for self-adjoint $A$. \cite{b16} considers the limit as $n\to\infty$ and provides a history of other such works.
\cite{b17}, followed by \cite{b18,b19,b20}, applies a decoupling argument to bound the small-ball probability of $x^*Ax$. Unfortunately, these arguments give insufficient control in this setting.

\subsection{Technical overview}\label{a5}
We consider a complex $n\times n$ matrix $A$ and it's compression $Q^*AQ$ onto an $\ell$ dimensional subspace. Throughout the paper, $Q\in\C^{n\times\ell}$ will be a complex $n\times\ell$ matrix with orthonormal columns and $q\in\C^n$ will be a complex unit vector.
Our primary result, stated in Theorem \ref{a2}, is several non-asymptotic bounds on $\E\area\Lambda_\eps(Q^*AQ)$. Simplifying the bounds under some of the assumptions (a), (b), and (c) results in Corollary \ref{a3}.

Section \ref{a6} combines standard ideas (nets and polarization) to provide a reduction from lower tail bounds on $\smin(z-Q^*AQ)$ to anti-concentration of a particular measure on $\C$. The resulting expression is of the form
\[\Pr\pare{\smin(z-Q^*AQ)\le\eps}\le\poly(\ell)\cdot\E\Pr\pare{\abs{q^*Sq}\le\poly(\ell)\cdot\eps\,|\,S}\]
where $S$ is a random Schur complement of $z-A$, and $q$ is a Haar-distributed unit vector (see Lemma \ref{a7}).
This section crucially uses that $Q$ is \textit{complex}; the analogous reduction over $\R$ is simply not true (take, for instance, $A$ to be skew-Hermitian and $z=0$).

Section \ref{a8} uses this reduction to produce a crude tail bound on $\smin(z-Q^*AQ)$ of the form
\spliteq{\label{a9}}{\Pr(\smin(z-Q^*AQ)\le\eps)\le C_{z,A}\cdot\eps}
where $C_{z,A}$ is a finite constant when $\sigma_{\ell+1}(z-A)$ is positive. This bound of $O(\eps)$ (as opposed to $O(\eps^2)$) on the right hand side is not strong enough to give control on $\area\Lambda_\eps(Q^*AQ)$ on the order of $\eps^\beta$ for $\beta>1$. It is, however, strong enough to give a baseline level of control that allows us then to obtain stronger estimates.

Section \ref{a10} analyzes anti-concentration of $q^*Mq$ for general fixed $M$. It applies standard inequalities about B-splines (appearing for instance in \cite{b21}) to an integral formula of \cite{b15} to reduce a rather complicated expression for the density to a much simpler one. From there, new arguments further bound the anti-concentration in terms of just two key quantities. The first is the area of the support of $q^*Mq$, which is simply the numerical range
\[W(M)=\set{q^*Mq:\magn q=1}.\]
The second is $\sigma_{9}(M)$, the ninth largest singular value of $M$. The resulting bound is of the form
\spliteq{\label{a11}}{
\Pr(\abs{q^*Mq}\le\eps)\le C_M\log^2(1/\eps)\cdot\eps^2
}
where $C_M$ is a finite constant when $\area W(M)>0$ and $\sigma_{9}(M)>0$ (see Proposition \ref{a12}).

Section \ref{a13} applies \eqref{a11} to the case where $M=S$ is a random Schur complement of $z-A$ in order to apply it to the reduction from Section \ref{a6}. 
For this bound to be effective, one must establish lower bounds on $\area W(S)$ and $\sigma_{9}(S)$. A couple different applications of the crude estimate \eqref{a9} establishes this. The resulting bound is of the form
\spliteq{\label{a14}}{
\Pr(\smin(z-Q^*AQ)\le\eps)\le C_{z,A}'\log^2(1/\eps)\cdot\eps^2
}
where $C_{z,A}'$ is finite when $\area W(A)>0$ and $\sigma_{\ell+8}(A)>0$ (see Proposition \ref{a15}).

Section \ref{a16} shows that the values of $C_{z,A}$ and $C'_{z,A}$ are small enough so that the tail bounds \eqref{a9} and \eqref{a14} imply effective bounds on $\E\area\Lambda_\eps(Q^*AQ)$.

\begin{center}
    \begin{tabular}{|l|l|}
        \hline
        \textbf{Notation} & \textbf{Meaning} \\
        \hline
        $\inR(\cdot)$ & Inner radius of a convex set\\
        $\semiu(m,k)$               & Set of $m\times k$ matrices with orthonormal columns.\\
        $X\sim\semiu(m,k)$          & $X$ is sampled from the Haar distribution on $\semiu(m,k)$.\\
        $(A/Q)$ & Generalized Schur complement $A-AQ(Q^*AQ)^{-1}Q^*A$ \\
        $\lambda_n(H)\le\cdots\le\lambda_1(H)$         & Eigenvalues of Hermitian $H$.\\
        $\sigma_n(M)\le\cdots\le\sigma_1(M)$               & Singular values of $M$.\\
        $\sigma_{\min}(M)$              & Least singular value $\sigma_{\min(n,m)}(M)$.\\
        $\Lambda_\eps(\cdot)$     & $\eps$-pseudospectrum.\\
        $W(\cdot)$     & Numerical range.\\
        \hline
    \end{tabular}
\end{center}
\section{Reduction to numerical measure}\label{a6}
The \textit{numerical measure} of a matrix $M$ is the measure $\mu_M$ on $\C$ such that
\[\mu_M(\Omega)=\Pr\pare{q^*Mq\in\Omega}\]
where $q\sim\semiu(n,1)$.
Given a matrix $A$ our goal is to find some (random) matrix $M$ such that that lower tail bounds on $\smin(Q^*AQ)$ can be derived from from expected anti-concentration of $\mu_M$.
It turns out this desired property is satisfied if $M$ is a random Schur complement of $A$.
We depart from standard notation slightly and define a Schur complement of $A$ as
\[(A/Q)=A-AQ(Q^*AQ)^{-1}Q^*A.\]
When $Q$ is the first few columns of the identity matrix, one recovers the usual definition of the Schur complement as the lower right block of $(A/Q)$.
A \textit{random} Schur complement is $(A/Q)$ where $Q\sim\semiu(n,\cdot)$.
The main idea to control $\smin(Q^*AQ)$ is to consider a polynomial sized net applied to $B=(Q^*AQ)^{-1}$.
We first design the net.
\begin{lemma}[Net construction]
\label{a17}
Let $\B=\set{e_j:j\in[\ell]}$ be an orthonormal basis of $\C^\ell$ and set\[\net=\B\cup\set{e_j+\omega^ae_k:j,k\in[\ell],j\neq k,a\in\set{0,1,2}}\]
for $\omega=e^{2\pi i/3}$. Then for all $B\in\C^{\ell\times\ell},$
\[\magn B\le \ell\cdot\max_{v\in\net}\abs{v^*Bv}.\]
\end{lemma}
\begin{proof}
Note $\magn B\le \ell\cdot\max_{j,k\in[\ell]}\abs{e_j^*Be_k}$. When the maximum is achieved for $j=k$, then note $e_j\in\net$ so certainly $\abs{e_j^*Be_j}\le\max_{v\in S}\abs{v^*Bv}$. When the maximum is achieved for $j\neq k$, apply the polarization identity:
\[y^*Bx=\frac{(x+y)^*B(x+y)+\omega(x+\omega y)^*B(x+\omega y)+\omega^2(x+\omega^2y)^*B(x+\omega^2y)}3.\]
Then by the triangle inequality,
\spliteq{}{
\abs{e_j^*Be_k}&\le\max\pare{
\abs{(e_j+e_k)^*B(e_j+e_k)},
\abs{(e_j+\omega e_k)^*B(e_j+\omega e_k)},
\rabs{(e_j+\omega^2 e_k)^*B(e_j+\omega^2 e_k)}
}\\&\le\max_{v\in\net}\abs{v^*Bv}
}
as required.
\end{proof}
Applying a union bound over this net gives the desired reduction
\begin{lemma}[Reduction to numerical measure]
\label{a7}
Fix any $A\in\C^{n\times n}$. Let $$Q\sim\semiu(n,\ell),\quad Q'\sim\semiu(n,\ell-1),\quad q\sim\semiu(n,1)$$ be independent and Haar distributed.
Then
\[\Pr\pare{\smin(Q^*AQ)\le\eps}\le3\ell^2\E\Pr\pare{\abs{q^*(A/Q')q}\le2\ell\eps\,|\, Q'}.\]
\end{lemma}
\begin{proof}
Let $\net$ be as defined in Lemma \ref{a17}. Then
that lemma immediately implies by union bound,
\spliteq{}{
\Pr\pare{\magn{(Q^*MQ)^{-1}}\ge\frac1\eps}
  \le
\Pr\pare{\max_{v\in S}\abs{v^*(Q^*MQ)^{-1}v}\ge\frac1{\eps\ell}}
  \le
\sum_{v\in S}\Pr\pare{\abs{v^*(Q^*MQ)^{-1}v}\ge\frac1{\eps\ell}}.
}
For each fixed $v\in\net$, denote $\hat v=\frac v{\magn v}$ and let $U_v$ be a unitary matrix carrying $e_1$ to $\hat v$. Then consider the block decomposition
$QU_v=\bmat{Q\hat v & Q'}$.
Then we may express
\spliteq{\label{a18}}{
\hat v^*(Q^*MQ)^{-1}\hat v
=\hat v^*U_v(U_v^*Q^*MQU_v)^{-1}U_v^*\hat v
=e_1^*\pare{\bmat{Q\hat v & Q'}^*M\bmat{Q\hat v & Q'}}^{-1}e_1.}
Then by the Schur formula for the inverse of a matrix,
\spliteq{\label{a19}}{
\eqref{a18}=\pare{
\hat v^*Q^*MQ\hat v
-
\hat v^*Q^*MQ'
(Q'^*MQ')^{-1}Q'MQ\hat v
}^{-1}=\pare{\hat v^*Q^*(M/Q')Q\hat v}^{-1}.
}
Since $U_v$ is fixed and $Q$ is Haar, $Q'\sim\semiu(n,\ell-1)$ is itself is Haar distributed. Conditioned on $Q'$, the vector $Q\hat v$ is Haar distributed on $\col(Q')^\perp$ and therefore can be expressed as
\[Q\hat v=\frac{(I-Q'Q'^*)q}{\magn{(I-Q'Q'^*)q}}\]
where $q$ is Haar distributed on the entire sphere and is independent of $Q'$. We combine this with \eqref{a19} by noting that $(M/Q')Q'=0$ and $Q'^*(M/Q')=0$. So we obtain simply
\[
\eqref{a18}=\pare{\frac{q^*(M/Q')q}{\magn{(I-Q'Q'^*)q}^2}}^{-1}.
\]
The last observation is that the denominator is at most 1 and $\magn v\le2$.
Altogether this gives
\spliteq{}{
\Pr\pare{\abs{v^*(Q^*MQ)^{-1}v}\ge\frac1{\eps\ell}}
=\Pr\pare{\frac{\magn{(I-(Q')(Q')^*)q}^2}{\abs{q^*(M/Q')q}}\ge\frac1{\magn v^2\eps\ell}}
\le\Pr\pare{q^*(M/Q')q\le2\eps\ell}.
}
Taking the union bound over $\abs\net\le3\ell^2$ terms and applying iterated expectation gives the result.
\end{proof}
\section{First order tail bound}
\label{a8}
\newcommand{\eitc}{e^{-i\theta}}
\newcommand{\eit}{e^{i\theta}}
The goal of this section is to derive a lower tail bound on $Q^*AQ$ with minimal assumptions on $A$. When $\rank(A)<\ell$, one deterministically has $\smin(Q^*AQ)=0$ for any $Q\in\semiu(n,\ell)$. This suggests a natural and necessary assumption that $\rank(A)\ge\ell$, and indeed this is the only assumption we require for our first tail bound.
Let
$$Q\sim\semiu(n,\ell),\quad Q'\sim\semiu(n,\ell-1),\quad q\sim\semiu(n,1).$$
Then, Lemma \ref{a7} shows the following implication.
\[
\text{Anti-concentration of }q^*(A/Q')q\quad\implies\quad\text{Lower tail bound on }\smin(Q^*AQ).\]
Note that there are two independent sources of randomness in $q^*(A/Q')q$, i.e. both $q$ and $Q'$. The bounds in this section allow an \textit{adversary} to pick $Q'$, so that we may focus on the anti-concentration of $q^*Mq$ for \textit{fixed} $M=(A/Q')$ and $q\in\semiu(n,1)$.
The main idea is to pass to the \textit{Hermitian part} of a multiple of $M$,
\[H(\eitc M):=\frac{\eitc M + \eit M^*}{2},\]
by making the key observation that for each $\theta$ and $x\in\C^n$,
\spliteq{\label{a20}}{
\abs{x^*Mx} 
\ge\abs{\Re\pare{\eitc x^*Mx}}
=\abs{\frac{\eitc x^*Mx+\eit x^*M^*x}2}
=\abs{x^*\frac{\eitc M+\eit M}2x}
=\abs{x^*H\pare{\eitc M}x}.
}
For a fixed Hermitian matrix $H$, we actually know the exact density of $q^*Hq$. Specifically,
$q^*Hq$ can be viewed as a random convex combination of the eigenvalues of $H$, where the weights are sampled from the Lebesgue measure on the simplex. This is exactly the measure defined by polynomial $B$-splines \cite{b22}.
The precise definition of $B$-splines is the following.
\begin{definition}\label{a21}
A $B$-spline with knots $t_1<\cdots<t_n$ is denoted $B[t_1,\ldots,t_n]$ and satisfies the recursive formula
\[B[t_j,t_{j+1}](t)=\mathbf1_{t\in(t_j,t_{j+1})},\]
\[B[t_j,\ldots,t_{j+\ell}](t)=\frac{t-t_{j}}{t_{j+\ell-1}-t_{j}}B[t_j,\ldots,t_{j+\ell-1}](t) +   \frac{t_{j+\ell}-t}{t_{j+\ell}-t_{j+1}} B[t_{j+1},\ldots,t_{n+\ell}](t).\]
\end{definition}
In particular, $B$ is a degree $n-2$ polynomial on each interval $[t_j,t_{j+1}]$, and is identically $0$ outside $(t_1,t_n)$.
Since $B$ is continuous in the knots $t_j$, the definition extends to weakly increasing knots $t_1\le\cdots\le t_n$ by taking the limit as adjacent knots approach each other. Several properties of $B$-splines follow directly from this definition; this section makes use of the following two.
\begin{fact}[\cite{b21}]\label{a22}
For any $B$-spline,
\[0\le B[t_1,\ldots,t_n](t)\le 1\quad\forall\,t\in\R,\]
\[\int_{-\infty}^\infty B[t_1,\ldots,t_n](t)=\frac{t_n-t_1}{n-1}.\]
\end{fact}
In particular, this makes
\spliteq{\label{a23}}{
\wt B[t_1,\ldots,t_n]:=\frac{n-1}{t_n-t_1}\cdot B[t_1,\ldots,t_n]
}
a probability density function bounded by $\frac{n-1}{t_n-t_1}$.
As we alluded to before, we concretely have the following lemma relating $B$-splines to the numerical measure of Hermitian matrices.
\begin{lemma}[Proposition 3.1 from \cite{b15}]\label{a24}
  If $H$ is Hermitian with eigenvalues $\lambda_n\le\cdots\le\lambda_1$, the probability density function of $q^*Hq$ is $\widetilde B[\lambda_n,\ldots,\lambda_1]$.
\end{lemma}

By applying the argument \eqref{a20} and Lemma \ref{a7}, the density bound on $q^*Hq$ from Lemma \ref{a24} is enough to derive the following tail bound on $\smin(Q^*AQ)$.
\begin{proposition}\label{a25}
  Fix $A\in\C^{n\times n}$ and $\eps>0$. Then for $Q\sim\semiu(n,\ell)$,
\[\Pr\pare{\smin(Q^*AQ)\le\eps}\le\constepso\frac{\eps}{\sigma_\ell(A)},\]
where $\constepso=24\ell^3(n-1)$.
\end{proposition}
\begin{proof}
By Lemma \ref{a7}, for $q\sim\semiu(n,1),Q'\sim\semiu(n,\ell-1).$
  \spliteq{}{
    \Pr\pare{\smin(Q^*AQ)\le\eps}
      &\le3\ell^2\E\Pr\pare{\abs{q^*(A/Q')q}\le2\ell\eps\,|\,Q' }
    \\&\le3\ell^2\sup_{U\in\semiu(n,\ell-1)}\Pr\pare{\abs{q^*(A/U)q}\le2\ell\eps}.
  }
Set $M=(A/U)$. There exists $u$ such that $\abs{u^*Mu}\ge\frac12\sigma_1(M)$. Let $\theta=\arg(u^*Mu)$. For any $x\in\C^n$,
\spliteq{\label{a26}}{
\abs{x^*Mx}
\ge\abs{\Re\pare{\eitc x^*Mx}}
=\abs{\frac{\eitc q^*Mq+\eit x^*M^*x}2}
=\abs{x^*\frac{\eitc M+\eit M}2x}
=\abs{x^*H\pare{\eitc M}x}}
with equality if $x=u$. Taking $x=u$ thus shows $\sigma_1\pare{H\pare{\eitc M}}\ge\frac12\sigma_1(M)$.
Since $M-A$ is rank $\ell-1$, we also have $\sigma_1(M)\ge\sigma_{\ell}(A)$. Combining these gives the bound
\spliteq{\label{a27}}{
  \Pr\pare{\smin(Q^*AQ)\le\eps}
    &\le3\ell^2\sup_{U\in\semiu(n,\ell-1)}\Pr\pare{\abs{q^*(A/U)q}\le2\ell\eps}
    \\&\le3\ell^2\sup_{U\in\semiu(n,\ell-1)}\Pr\pare{\abs{q^*H(\eitc(A/U))q}\le2\ell\eps}
    \\&\le3\ell^2\sup_{H,\,\sigma_1(H)\ge\frac12\sigma_\ell(A)}\Pr\pare{\abs{q^*Hq}\le2\ell\eps}
}
where the supremum is over all Hermitian $H$ satisfying $\sigma_1(H)\ge\frac12\sigma_\ell(A)$. From Fact \ref{a22} and Lemma \ref{a24}, the density of $q^*Hq$ is bounded by $\frac{n-1}{\lambda_1(H) - \lambda_n(H)}$. It's supported on $[\lambda_n(H),\lambda_1(H)]$, and we want to know the probability it lands in the interval $[-2\ell\eps,2\ell\eps]$. So we can integrate the density over the intersection of those intervals to obtain
\spliteq{\label{a28}}{
\Pr\pare{\abs{q^*Hq}\le2\ell\eps}
\le
\pare{\min(2\ell\eps, \lambda_1(H)) - \max(-2\ell\eps,\lambda_n(H))}
\cdot \frac{n-1}{\lambda_1(H) - \lambda_n(H)}.
}
For each fixed nonzero $H$, we must either have $\lambda_1(H)>0$ or $\lambda_n(H)<0$.
Suppose $\lambda_1(H)>0$. Thinking of $\lambda_1(H)$ as fixed, the right hand side of \eqref{a28} is maximized for $\lambda_n(H)=-2\ell\eps$. Thus one obtains
\[\Pr\pare{\abs{q^*Hq}\le2\ell\eps}
  \le4\ell\eps\cdot\frac{n-1}{\lambda_1(H)}.\]
Now suppose $\lambda_n(H)<0$. Then a symmetric calculation gives
\[\Pr\pare{\abs{q^*Hq}\le2\ell\eps}
  \le4\ell\eps\cdot\frac{n-1}{-\lambda_n(H)}.\]
When both $\lambda_n<0<\lambda_1$, both bounds apply. Note $\sigma_1(H)=\max(\lambda_1(H),-\lambda_n(H))$, so altogether this gives the unconditional bound
\[\Pr\pare{\abs{q^*Hq}\le2\ell\eps}
  =  4\ell\eps\cdot\frac{n-1}{ \sigma_1(H) }
  \le8\ell\eps\cdot\frac{n-1}{\sigma_\ell(A)}.\]
Plugging this into \eqref{a27} gives the desired result.
\end{proof}

\section{Numerical measure}\label{a10}
\newcommand{\wmax}{w_1^{\max}}
\newcommand{\wmin}{w_1^{\min}}
Section \ref{a8} obtained tail bounds on $\smin(Q^*AQ)$ by considering the numerical measure of some Hermitian matrices. The result was of the form
\[\Pr\pare{\smin(Q^*AQ)\le\eps}\le C_A\eps\]
under only the assumption that $\rank(A)\ge\ell$. We want to upgrade $\eps$ to $\eps^2$. Unfortunately, this is not always possible, e.g. when $A$ is a multiple of a Hermitian matrix. In the next section we obtain nearly the desired upgrade (i.e. up to a $\poly\log(1/\eps)$ factor) under the additional assumption that $W(A)$ has nonempty interior, and the strengthened rank assumption $\rank(A)\ge\ell+8$.

To obtain this upgrade, we need to more carefully consider the numerical measure of non-Hermitian matrices; this is the aim of this section.
Unfortunately, there are non-Hermitian matrices $M$ such that the density of $q^*Mq$ cannot be uniformly bounded as it may contain singularities. We derive sufficient conditions under which the singularities ``weak'' enough so that we still obtain strong small-ball estimates.

Subsection \ref{a29} provides a convenient bound on the density of $q^*Mq$. Subsection \ref{a30} converts this to a finite small-ball estimate.%

\subsection{Density}\label{a29}
Implicit in the observation \eqref{a20} is the following: the projection of the distribution of $q^*Mq$ onto the line $\eit\R$ rotated back to the real axis is exactly the distribution of $q^*H(\eitc M)q$. Computing this projection for every $\theta$ is called \textit{Radon transform} of the numerical measure. Remarkably, there's an inversion formula for the Randon transform involving the Hilbert transform $\H$,
\spliteq{\label{a31}}{
\H(f)(t)=\frac1\pi\textnormal{ p.v.}\int_{-\infty}^\infty\frac{f(\tau)}{t-\tau}\wrt\tau}
where $\text{p.v.}$ denotes taking the Cauchy principle value.
The result of its application is the following formula for the density appearing in \cite{b15}.
\begin{lemma}[Theorem 4.2 of \cite{b15}]\label{a32}
Fix $M\in\C^{n\times n}$ and let $q\sim\semiu(n,1)$.
Let $\rho_M(z)$ be the probability density function of $q^*Mq$
and $\rho_\theta(z)$ the probability density function of $q^*H(\eitc M)q$.
Then
\[\rho(z)=\frac1{4\pi}\textnormal{ p.v.}\int_{0}^{2\pi}\H(\rho'_\theta)(\Re(e^{-i\theta}z))\wrt\theta.\]
\end{lemma}

Our first main contribution is bounding $\H(\rho'_\theta)$ in terms of the following simpler quantities.
\begin{definition}\label{a33}
For Hermitian $H\in\C^{n\times n}$, let
\spliteq{}{
w_1(H)&=\lambda_1(H)-\lambda_n(H)
\\
w_2(H)&=\pare{\pare{\lambda_2(H)-\lambda_n(H)}^{-1}+\pare{\lambda_1(H)-\lambda_{n-1}(H)}^{-1}}^{-1},
\\
w_3(H)&=\lambda_2(H)-\lambda_{n-1}(H).}
\end{definition}
We've already encountered $w_1(H)$, specifically in the denominator of $\wt B$ \eqref{a23} in Lemma \ref{a24}. $w_2(H)$ and $w_3(H)$ are motivated by the recursive formula for the derivative of $B$-splines, stated below.
\begin{fact}[\cite{b21}]\label{a34}
  For any $B$-spline with at least 3 knots,
  \item\[\deriv t{}B[t_1,\ldots,t_n](t)=(n-2)\pare{ \frac{B[t_1,\ldots,t_{n-1}](t)}{t_{n-1}-t_1} - \frac{B[t_2,\ldots,t_n](t)}{t_n-t_2} }\]
\end{fact}
The important implication of this toward bounding $\H(\rho'_\theta)$ is the following concavity estimate for the numerical measure of Hermitian matrices.
\begin{lemma}\label{a35}
Fix Hermitian $H\in\C^{n\times n},n\ge4$ and let $q\sim\semiu(n,1)$. Let $\rho:\R\to\R$ be the density function of $q^*Hq$. Then 
\[\deriv t2\rho(t)\ge
  \begin{cases}
    -\frac{(n-1)(n-2)(n-3)}{w_1(H)w_2(H)w_3(H)} & t\in[\lambda_2,\lambda_{n-1}]
    \\
    0& t\not\in[\lambda_2,\lambda_{n-1}]
  \end{cases}.
  \]
\end{lemma}
\begin{proof}
By applying Fact \ref{a34} twice and using $0\le B\le 1$ from Fact \ref{a22}, one obtains
  \begin{equation}
    \deriv t2 B[t_1,\cdots,t_n](t)\ge
    \begin{cases}
      -\frac{(n-2)(n-3)}{t_{n-1}-t_2}\pare{ \frac1{t_{n-1}-t_1} + \frac1{t_n-t_2} } & t\in[t_2,t_{n-1}]
      \\
      0& t\not\in[t_2,t_{n-1}]
    \end{cases}.
  \end{equation}
  By Lemma \ref{a24}, the left hand side is exactly $\deriv t2\rho(t)$ scaled by $w_1(H)/(n-1)$ for $t_j=\lambda_{n-j+1}(H)$. Then the right hand side can be expressed in terms of $w_j(H)$ just by using their definition, Definition \ref{a33}.
\end{proof}
Some casework allows us to bound $\H(\rho_\theta')$ using this concavity estimate.
\begin{lemma}\label{a36}
Fix Hermitian $H\in\C^{n\times n},n\ge 4$ and let $q\sim\semiu(n,1)$.
Let $\rho:\R\to\R$ be the density function of $q^*Hq$. Then for $t\in[\lambda_n,\lambda_1]$,
  \[\H(\rho')(t)\le\frac1{\pi}\cdot\frac{(n-1)(n-2)(n-3)}{w_1(H)w_2(H)} \log\pare{4e\cdot\frac{w_1(H)}{w_3(H)}}.\]
\end{lemma}
\begin{proof}
We start with the case of $t=0$.
Let $\lambda_n\le\ldots\le\lambda_1$ be the eigenvalues of $H$.
Then by definition of the Hilbert transform \eqref{a31},
\[
\H(\rho')(0)
=\frac1\pi\text{ p.v.}\int_{-\infty}^\infty\frac{-\rho'(\tau)}{\tau}\wrt\tau.\]
Observe that $f\mapsto \H(f)(0)$ is a linear map and $\H(f)(0)=0$ whenever $f$ is an even function. Set
\[f(x)=\rho'(0)\cdot\mathbf1_{\abs x\le\magn H}.\]
Note also that $\rho'(\tau)=0$ for $\tau$ outside of $[\lambda_n,\lambda_1]\subset[-\magn H,\magn H]$. Using those facts along with the fundamental theorem of calculus,
\spliteq{}{
 \H(\rho')(0)
  &=\H(\rho'-f)(0)
\\&=\frac1\pi\text{ p.v.}\int_{-\infty}^\infty\frac{f(\tau)-\rho'(\tau)}{\tau}\wrt\tau
\\&=\frac1\pi\text{ p.v.}\int_{-\magn H}^{\magn H}\frac{\rho'(0)-\rho'(\tau)}{\tau}\wrt\tau
\\&=\frac1\pi\text{ p.v.}\int_{-\magn H}^{\magn H}\pare{\frac1\tau\int_0^{\tau}-\rho''(x)\wrt x}\wrt\tau
.}
We can now apply the concavity estimate from Lemma \ref{a35} on $\rho$. In particular, $-\rho''$ is nonpositive when $t$ is outside of $[\lambda_{n-1},\lambda_2]$ and otherwise we have the upper bound
\[\sup_{x\in\R}(-\rho''(x))\le\frac{(n-1)(n-2)(n-3)}{w_1(H)w_2(H)w_3(H)}=:S.\]
By assumption, $0\in[\lambda_n,\lambda_1]$. We consider three cases based which of the three intervals
\[[\lambda_n,\lambda_{n-1}]\sqcup(\lambda_{n-1},\lambda_2)\sqcup [\lambda_2,\lambda_1]\]contains $0$.
The first case we consider is $\lambda_{n-1}\le0\le\lambda_2$. We integrate to the left and right of 0 separately.
Then
\spliteq{\label{a37}}{
  \int_0^{\magn H}\pare{\frac1\tau\int_0^\tau-\rho''(x)\wrt x}\wrt \tau
&=
  \int_0^{\magn H}\pare{\frac1\tau\int_0^{\min(\tau,\lambda_2)}-\rho''(x)\wrt x}\wrt \tau
  \\&=
  S\cdot\int_0^{\magn H}\frac{\min(\tau,\lambda_2)}\tau\wrt \tau
  \\&=
  S\cdot\pare{\lambda_2 + \lambda_2\log\pare{\frac{\magn H}{\lambda_2}}}
  \\&=
  S\cdot\lambda_2\log\pare{\frac{e\magn H}{\lambda_2}}
}
and
\spliteq{\label{a38}}{
  \int_{-\magn H}^0\pare{\frac1\tau\int_0^\tau-\rho''(x)\wrt x}\wrt \tau
&=
  \int_{-\magn H}^0\pare{\frac1\tau\int_0^{\max(\tau,\lambda_{n-1})}-\rho''(x)\wrt x}\wrt \tau
  \\&=
  S\cdot\int_{-\magn H}^0\frac{\max(\tau,\lambda_{n-1})}\tau\wrt \tau
  \\&=
  S\cdot\pare{-\lambda_{n-1} - \lambda_{n-1}\log\pare{\frac{\magn H}{-\lambda_{n-1}}}}
  \\&=
  S\cdot(-\lambda_{n-1})\log\pare{\frac{e\magn H}{-\lambda_{n-1}}}.
}
Note $\lambda_2+\lambda_{n-1}=w_3(H)$. So we have
\[
  \eqref{a37}+\eqref{a38}\le\sup_{ \substack{a<0<b\\ a+b=w_3(H)}  }
  S\cdot b\log\pare{\frac{e\magn H}{b}}
  +
  S\cdot(-a)\log\pare{\frac{e\magn H}{-a}}.
\]
By taking derivatives, one can see that this is maximized for $-a=b=\frac12w_3(H)$.
Since $\lambda_n\le0\le\lambda_1$, we furthermore have the bound $\magn H\le w_1(H)$.
Thus we obtain in the case of $\lambda_{n-1}\le0\le\lambda_2$ that
\spliteq{\label{a39}}{
  \H(\rho')(0)
  \le\frac S{\pi}\cdot w_3(H)\log\pare{2e\cdot \frac{w_1(H)}{w_3(H)}}
  =\frac1\pi\cdot\frac{(n-1)(n-2)(n-3)}{w_1(H)w_2(H)}\log\pare{2e\cdot\frac{w_1(H)}{w_3(H)}}.}
Now we look at the case of $\lambda_n\le 0\le \lambda_{n-1}$. Then the integral becomes
\spliteq{}{
\int_{-\magn H}^{\magn H}\pare{\frac1\tau\int_0^\tau-\rho''(x)\wrt x}\wrt\tau
  &\le\int_{\lambda_{n-1}}^{\magn H}\pare{\frac1\tau\int_{\lambda_{n-1}}^{\min(\lambda_2,\tau)}-\rho''(x)\wrt x}\wrt\tau
  \\&\le S\cdot\int_{\lambda_{n-1}}^{\magn H}\pare{\frac{\min(\lambda_2,\tau)-\lambda_{n-1}}\tau}\wrt\tau
  \\&= S\cdot\pare{
    \int_{\lambda_{n-1}}^{\lambda_2}\pare{\frac{\min(\lambda_2,\tau)-\lambda_{n-1}}\tau}\wrt\tau
    +
    \int_{\lambda_2}^{\magn H}\pare{\frac{\min(\lambda_2,\tau)-\lambda_{n-1}}\tau}\wrt\tau
  }
  \\&= S\cdot\pare{
    \int_{\lambda_{n-1}}^{\lambda_2}\pare{\frac{\tau-\lambda_{n-1}}\tau}\wrt\tau
    +
    \int_{\lambda_2}^{\magn H}\pare{\frac{\lambda_2-\lambda_{n-1}}\tau}\wrt\tau
  }
  \\&\le S\cdot\pare{
    \int_{\lambda_{n-1}}^{\lambda_2}1\wrt\tau
    +
    w_3(H)\int_{\lambda_2}^{\magn H}{\frac1\tau}\wrt\tau
  }
  \\&=S\cdot w_3(H)\log\pare{\frac{e\magn H}{\lambda_2}}
  \\&\le S\cdot w_3(H)\log\pare{\frac{e\magn H}{w_3(H)}}.
}
Note this is upper bounded by the right hand side of \eqref{a39}.
A symmetrical calculation gives exactly the same bound for $\lambda_2\le 0\le\lambda_1$, so \eqref{a39} holds in all three cases.
This finishes the proof for $t=0$. For other $t$, apply the above argument to the matrix $\wt H = H-tI$. Note that since $w_j(\wt H)=w_j(H)$, the right hand side of \eqref{a39} does not change if $H$ is replaced by $\wt H$.
Let $\wt \rho(x)=\rho(x+t)$ be the density function of $q^*\wt Hq$.
Then
\[
\H(\rho')(t)
=\frac1\pi\text{ p.v.}\int_{-\infty}^\infty\frac{\rho'(\tau)}{t-\tau}
=\frac1\pi\text{ p.v.}\int_{-\infty}^\infty\frac{\rho'(\tau+t)}{-\tau}
=\frac1\pi\text{ p.v.}\int_{-\infty}^\infty\frac{\wt\rho'(\tau)}{-\tau}
=\H(\wt\rho')(0).\]
Then apply \eqref{a39} to $\H(\wt\rho')(0)$ to obtain the final result.
\end{proof}
If we denote for concision $w_j(\theta)=w_j(H(\eitc M)),$ this results in the following density bound.
\begin{proposition}\label{a40}
Fix any $M\in\C^{n\times n},n\ge4$ and let $q\sim\semiu(n,1).$ Let $\rho_M$ be the density function of $q^*Mq$. Then
\[
\sup_{z\in\C}\rho_M(z)\le\frac{(n-1)(n-2)(n-3)}{4\pi^2}\int_0^{2\pi}\frac1{w_1(\theta)w_2(\theta)}\log\pare{4e\cdot\frac{w_1(\theta)}{w_3(\theta)}}
\]
\end{proposition}
\begin{proof}
Set $H(\theta)=H(\eitc M)$ and let $\rho_\theta$ be the density of $q^*H(\theta)q$. We first apply Lemma \ref{a24},
\[\rho_M(z)=\frac1{4\pi}\int_0^{2\pi}\H(\rho_\theta')(\Re(\eitc z))\wrt\theta.\]
Note for each $\theta$ that $W(M)$ is contained in the strip
\[W(M)\subset\set{z\in\C:\lambda_n(H(\theta))\le\Re(\eitc z)\le\lambda_1(H(\theta))}.\]
If there is any $\theta$ for which
\[\Re(\eitc z)\not\in\sqbrac{\lambda_n(H(\theta)), \lambda_1(H(\theta))},\]
then $z\not\in W(M)$ and so $\rho_M(z)=0$.
We can therefore assume this does not occur, so the requirement of Lemma \ref{a36} applies to $H(\theta)$ for every $\theta$, immediately giving the required bound on $\H(\rho_\theta').$
\end{proof}

\subsection{Small-ball probability}\label{a30}
This section converts the density $\rho_M(z)$ bound on $q^*Mq$ into a small-ball probability bound on $q^*Mq$. One way to turn densities into small-ball estimates is to integrate a density bound over the relevant region on $\C$. Doing this with Proposition \ref{a40} isn't strong enough because there are matrices for which the bound is infinity.
We address this issue by making the following modification.
Let \[N=\bmat{0&2\\0&0}.\]
Throughout this section, we fix a particular $M\in\C^{n\times n}$ and set
\[M'=M\oplus \eps N\oplus \eps N\]
where $\oplus$ denotes the direct sum. We also let $q\sim\semiu(n,1)$ and $q'\sim\semiu(n+4,1)$.
Replacing $M$ with $M'$ provides enough regularization to produce a bound that's always finite. We will obtain a small-ball probability by passing to the density of $q'^*M'q'$ instead of $q^*Mq$. When we apply Proposition \ref{a40}, this density will be in terms of $w_j(\cdot)$, so for concision we denote
\[H(\theta)=H\pare{\eitc M'}\qand w_j(\theta)=w_j(H(\theta))\]
for the remainder of this section.
The first lemma of this subsection shows how to pass from $M$ to $M'$.
\begin{lemma}\label{a41}
Let $\rho_{M'}$ be the probability density function of $q'^*M'q'$. Then
\[\Pr\pare{\abs{q^*Mq}\le\eps}\le\pi\eps^2\sup_{z\in\C}\rho_{M'}(z).\]
\end{lemma}
\begin{proof}
We may couple $q,q'$ via
\[q'=\bmat{aq\\br\\cs}\]
where $r,s\sim\semiu(2,1)$, and $a,b,c$ are positive real random variables such that $a^2+b^2+c^2=1$. Then
\[q'^*M'q'=a^2q^*Mq + b^2r^*(\eps N)r + c^2s^*(\eps N)s\]
is a convex combination of $q^*Mq,r^*(\eps N)r,s^*(\eps N)s$.
The latter two terms are elements of $W(\eps N)=\ball(0,\eps)$, which is convex.
Thus if $q^*Mq\in\ball(0,\eps)$ then $q'^*M'q\in\ball(0,\eps)$ as well, i.e.
\[\Pr\pare{\abs{q^*Mq}\le\eps}\le\Pr\pare{\abs{q'^*M'q'}\le\eps}=\int_{\abs z\le\eps}\rho_{M'}(z)\wrt z\le\pi\eps^2\sup_{\abs z\le\eps}\rho_{M'}(z).\]
\end{proof}
When we apply Proposition \ref{a40}, we will apply Holder's inequality to the integral so that the factors we must upper bound are
\[
\text{``}L^\infty\text{ factor'' }=
\sup_{\theta\in[0,2\pi)}\log\pare{4e\cdot\frac{w_1(\theta)}{w_3(\theta)}}
\qand
\text{``}L^1\text{ factor'' }=
\int_0^{2\pi}\frac1{w_1(\theta)w_2(\theta)}\wrt\theta.\]
Controlling the $L^\infty$ factor can be done with a short argument.
\begin{lemma}\label{a42}
    \[\sup_{\theta\in[0,2\pi)}\frac{w_1(\theta)}{w_3(\theta)}\le\frac{\magn{M'}}\eps.\]
\end{lemma}
\begin{proof}
Fix any $\theta$.
By definition of $w_1(\cdot)$, we have by the triangle inequality
\spliteq{}{
w_1(\theta)
=\lambda_1(H(\theta))-\lambda_n(H(\theta))
\le2\magn{H(\theta)}.}
The norm $\magn{H(\theta)}$ is bounded via
\[
\magn{H(\theta)}
\le\sup_{\magn x=1}x^*H(\eitc M')x
\le\sup_{\magn x=1}\Re(\eitc x^*M'x)
\le\sup_{\magn x=1}\abs{x^*M'x}
\le\magn{M'}\]
so $w_1(\theta)\le2\magn{M'}$. Next, the spectrum of
\[H(\theta)=H(\eitc M)\oplus H(\eitc \eps N)\oplus H(\eitc \eps N)\]
contains two copies of spectrum of $H(\eitc \eps N)$, which is exactly the set $\set{-\eps,\eps}$. Therefore,
\[
\lambda_{n-1}(H(\eitc M'))\le-\eps\qand\eps\le\lambda_2(H(\eitc M'))\quad\implies\quad
w_3(H(\eitc M'))\ge2\eps.\]
The result follows by taking the quotient of $2\magn{M'}$ and $2\eps$.
\end{proof}
For the $L^1$ factor, more work is required. We introduce the following functional which will serve as an intermediate.
\spliteq{\label{a43}}{
\phi(M')&=\sup_{U\in\semiu(n+4,2)}\Phi(U^*M'U),\\
\Phi\pare{\bmat{a&b\\c&d}}&=\abs{ad-bc}^2-\abs{ \Re(\overline ad) - \frac{\abs b^2}2 - \frac{\abs c^2}2 }^2.}
We bound the $L^1$ factor in terms of $\phi(M')$.
\begin{lemma}\label{a44}
\[
\int_0^{2\pi}\frac1{w_1(\theta)w_2(\theta)}\wrt\theta
\le
\frac{4\pi+16\log\pare{{12}^{1/4}\cdot\magn{M'}/\eps}}{\abs{\phi(M')}^{1/2}}.\]
\end{lemma}
\begin{proof}
For concision, we drop the explicit dependence on $\theta$ and let $\lambda_j$ be the eigenvalues of $H(\theta)$ and $\sigma_1,\sigma_2$ be the largest two singular values of $H(\theta)$.
Then
\spliteq{}{
w_1(\theta)w_2(\theta)
&=
(\lambda_1-\lambda_n)
\cdot
\pare{
(\lambda_1-\lambda_{n-1})^{-1}
+
(\lambda_2-\lambda_n)^{-1}
}^{-1}
\\&\ge\frac12\cdot(\lambda_1-\lambda_n)
\cdot\min\pare{
\lambda_1-\lambda_{n-1},\,
\lambda_2-\lambda_n}
\\&\ge\frac12\cdot\max(\abs{\lambda_1},\abs{\lambda_n})
\cdot\min\pare{
\lambda_1-\lambda_{n-1},\,
\lambda_2-\lambda_n}
\\&=\frac12\cdot\sigma_1
\cdot\min\pare{
\lambda_1-\lambda_{n-1},\,
\lambda_2-\lambda_n}
.}
By construction of $M'$, we have $\lambda_{n-1}<0<\lambda_2$ so
\[
\min\pare{
\lambda_1-\lambda_{n-1},\,
\lambda_2-\lambda_n
}\ge\max\pare{
\abs{\lambda_2},
\abs{\lambda_{n-1}},
\min(\abs{\lambda_1},\abs{\lambda_n})}=\sigma_2.\]
For any $U\in\semiu(n+4,2)$, by interlacing, $\sigma_1,\sigma_2$ dominate the singular values of $U^*H(\theta)U$, the product of which is simply the absolute determinant. That is,
\[2w_1(\theta)w_2(\theta)\ge\sup_{U\in\semiu(n+4,2)}\abs{\det(U^*H(\theta)U)}=\sup_{U\in\semiu(n+4,2)}\abs{\det(H(\eitc U^*M'U))}.\]
By taking $U=\bmat{e_{n+1}&e_{n+2}}$, one has \(U^*M'U=\eps N\) so $2w_1(\theta)w_2(\theta)\ge\eps^2$. Now take $U$ so that $\Phi(U^*M'U)=\phi(M')$. Set
\[\bmat{a&b\\c&d}=U^*M'U.\]
Then there exists some angle $\theta_0$ such that
\begin{align*}
2w_1(\theta)w_2(\theta)\ge\abs{\det\pare{H\pare{e^{-i\theta} \bmat{a&b\\c&d} }}}
  &=\abs{
\Re\pare{e^{-2i\theta}(ad-bc)}+\Re\pare{\conj ad}
-\frac{\abs{b}^2}2
-\frac{\abs{c}^2}2}
\\&=\abs{
\abs{ad-bc}\cos(2\theta-\theta_0)+\Re\pare{\conj ad}-\frac{\abs{b}^2}2-\frac{\abs{c}^2}2}.
\end{align*}
Set $\alpha=\abs{ad-bc}$ and $\beta=\Re\pare{\conj ad}-\frac{\abs{b}^2}2-\frac{\abs{c}^2}2$ so that $\phi(M')=\alpha^2-\beta^2$. 
Then\[
\int_0^{2\pi}\frac1{w_1(\theta)w_2(\theta)}\wrt\theta
\le
\int_0^{2\pi}\frac1{\max\pare{\eps^2,\abs{\beta+\alpha\cos(2\theta-\theta_0)}}}\wrt\theta.\]
This integral calculation is performed in Lemma \ref{a45} in the appendix and its result is
\[
\int_0^{2\pi}\frac1{w_1(\theta)w_2(\theta)}\wrt\theta
\le
\frac{4\pi+16\log\pare{\sqrt2\cdot{\abs{\phi(M')}^{1/4}}/{\eps}}}{\abs{\phi(M')}^{1/2}}
.\]
Finally, $\phi(M')\le 3\magn{U^*M'U}^4\le3\magn{M'}^4$.
\end{proof}

Our next goal is to compute some $U\in\semiu(n+4,2)$ such that $\Phi(U^*MU)$ is lower bounded. To do so, we need the following technical lemma about the curvature of a cone.

\begin{lemma}\label{a46}
Let $f(x,y,z)=x^2+y^2-z^2$.
Say $(x,y),(x',y')\in\R^2$ are the vertices of a non-obtuse isosceles triangle with apex at the origin. Let $d$ be the length of it's base.
Then for any $z,z'\in\R$ there exists a point $u\in\R^3$ on the line segment connecting $(x,y,z),(x',y',z')$ with\[\abs{f(u)}\ge\frac{d^2}8.\]
\end{lemma}
\begin{proof}
By rotating $(x,y,z),(x',y',z')$ about the $z$-axis, we may assume they have the form
\(\pare{x,\pm\frac d2,z_\pm}.\)
The midpoint between these endpoints is $\pare{x,0,\frac{z_++z_-}2}$. Suppose
\[\abs{f\pare{x,\pm\frac d2,z_\pm}}<\frac{d^2}8\qand \abs{f\pare{x,0,\frac{z_++z_-}2}}<\frac{d^2}8.\]
Then there are six linear constraints on $x^2$,
\[
\begin{rcases}
    z_+^2-\pare{\frac d2}^2-\frac{d^2}8
    \\
    z_-^2-\pare{\frac d2}^2-\frac{d^2}8
    \\
    \pare{\frac{z_++z_-}2}^2-\frac{d^2}8
\end{rcases}
<x^2<
\begin{lcases}
    z_+^2-\pare{\frac d2}^2+\frac{d^2}8
    \\
    z_-^2-\pare{\frac d2}^2+\frac{d^2}8
    \\
    \pare{\frac{z_++z_-}2}^2+\frac{d^2}8
\end{lcases}.\]
Combining each of the first two lower bounds with the first two upper bounds gives
\[
z_\pm^2-\frac{d^2}8<z_\pm^2+\frac{d^2}8
\implies
\abs{\frac{z_++z_-}2}<\frac{d^2/8}{\abs{z_+-z_-}}.\]
Combining this with the third upper bound gives
\[
x^2
<\pare{\frac{z_++z_-}2}^2+\frac{d^2}8
<\pare{\frac{d^2/8}{\abs{z_+-z_-}}}^2+\frac{d^2}8
\]
The isosceles condition ensures $\frac{d^2}4\le x_0^2$, so we have
\[
\frac{d^2}{8}<\pare{\frac{d^2/8}{\abs{z_+-z_-}}}^2
\implies 
\abs{z_+-z_-}<\frac d{\sqrt 8}.\]
Combining the isosceles condition with the first two upper bounds results in
\[
\frac{d^2}4<z_\pm^2-\frac{d^2}{8}
\implies
\frac{d}{\sqrt{8/3}}<\abs{z_\pm}.\]
In particular, we conclude that $z_+$ and $z_-$ must have the same sign. Without loss of generality, say they are both positive. That means
\[
\min(z_+,z_-)^2
\le
\pare{\frac{z_++z_-}2}^2.
\]
Combine the first two upper bounds on $x^2$ with the third lower bound on $x^2$. Then
\[
\pare{\frac{z_++z_-}2}^2-\frac{d^2}8
<\min(z_+,z_-)^2-\frac{d^2}{8}
\le\pare{\frac{z_++z_-}2}^2-\frac{d^2}{8}
\]
which is a contradiction.
\end{proof}

We're now ready for the main bound on $\phi(\cdot)$.
\begin{lemma}\label{a47}
\[\abs{\phi(M')}^{1/2}\ge\frac{\sigma_9(M')\inR(W(M'))}4.\]
where $\inR(\cdot)$ is the inner radius.
\end{lemma}
\begin{proof}
Suppose $z_0$ is such that $W(M')\supset\ball(z_0,\inR(W(M')))$. Then
\[z_\pm=z_0+\frac{z_0}{\abs{z_0}}\cdot\frac{1\pm i}{\sqrt 2}\cdot\inR(W(M'))\]
are the vertices of a non-obtuse isosceles triangle with apex at $0$ and base length $\sqrt2\inR(W(M'))$.
Let $v,v'$ be unit vectors such that $v^*M'v=z_+$ and $v'^*M'v'=z_-$.
Then set $V_1\in\semiu(n+1,3)$ and $V_2\in\semiu(n+1,m)$ such that
\[
\col(V_1)=\spn\set{v,v',e_{n+1}}
\qand
\col(V_2)=\spn\set{v,v',e_{n+1},M'v,M'v'}^\perp.\]
Let $y\in\C^m$ be such that $\abs{y^*V_2^*MV_2y}\ge\frac12\magn{V_2^*MV_2}$. Set
$$s=y^*V_2^*M'V_2y\qand u^*=y^*V_2^*M'V_1\qand B=V_1^*M'V_1.$$
Then consider the following, thought of as a function of $x\in\C^2,\magn x=1$,
\[\bmat{V_2y & V_1x}^*M\bmat{V_2y & V_1x}=\bmat{s & u^*x \\ & x^*Bx}.\]
In particular, by setting $U_x=\bmat{V_2y&V_1x}$, see that
\[\phi(M)\ge\Phi(U_x^*MU_x)=\abs{\conj s x^*Bx}^2-\abs{\Re\pare{\conj sx^*Bx}^2-\frac{\abs{u^*x}^2}2}^2.\]
We can express this as the composition of three functions.
\begin{align}
T&:\set{x\in\C^3:\magn x=1}\to\R^3,\quad T(q)
=
\bmat{
\Re(\conj sx^*Bx)\\
\Im(\conj sx^*Bx)\\
{\abs{u^*x}^2}/2
}
=
\bmat{
x^*H(\conj sB)x\\
x^*H(-i\conj sB)x\\
x^*\frac{uu^*}2x
}
\\
S&:\R^3\to\R^3,\quad S(x,y,z)=(x,y,x-z)
\\
f&:\R^3\to\R,\quad f(x,y,z)=x^2+y^2-z^2.
\end{align}
That is,
\[\Phi(U_x^*M'U_x)=f(S(T(x))).\]
Note the coordinates of $T$ are elements of numerical ranges of Hermitian matrices, so it's range is a generalized numerical range and is convex (see Proposition 3 of \cite{b23}). The shadow of $\range(T)$ on the $xy$-plane treated as $\C$ is simply the numerical range $W(\conj s B)$. Then see that $S$ is a linear map which preserves this shadow. By construction, $W(\conj s B)=\conj sW(V_1^*M'V_1)$ contains $\conj sz_\pm$, i.e. the vertices of a non-obtuse isosceles triangle with apex at 0 and base length at least $$\abs{s}\sqrt 2\inR(W(M'))\ge\frac1{\sqrt2}\magn{V_2^*M'V_2}\inR(W(M')).$$
Thus by Lemma \ref{a46}, we are guaranteed a point $p\in S(\range(T))$ with
\[\abs{f(p)}\ge\frac1{16}\magn{V_2^*M'V_2}^2\inR(W(M'))^2\]
Finally, note the dimension of $\col(V_2)\cap\spn\set{e_1,\cdots,e_n}$ is at least $n-4$, so $\magn{V_2^*M'V_2}\ge\sigma_9(M')$.
\end{proof}

Everything culminates in the following small-ball probability bound.
\begin{proposition}[Small-ball probability bound for numerical measures]\label{a12}For any nonzero $M\in\C^{n\times n}$ and $\eps\in(0,\magn M)$,
\[\Pr(\abs{q^*Mq}\le\eps)
\le\eps^2\log^2\pare{4e\magn{M}/\eps}
\cdot\frac{5.1\,(n+3)^3}{\sigma_9(M)\inR(W(M))}.\]
\end{proposition}
\begin{proof}
    Start by applying Lemma \ref{a41},\[(\star)=\Pr\pare{\abs{q^*Mq}\le\eps}\le\pi\eps^2\sup_{z\in\C}\rho_{M'}(z).\]
    Then apply Proposition \ref{a40} to $\rho_{M'}$,
    \[
    (\star)\le\pi\eps^2\sup_{z\in\C}\rho_{M'}(z)\le\pi\eps^2\cdot\frac{(n+3)(n+2)(n+1)}{4\pi^2}\int_0^{2\pi}\frac1{w_1(\theta)w_2(\theta)}\log\pare{4e\cdot\frac{w_1(\theta)}{w_3(\theta)}}\wrt\theta.
    \]
    Now apply Lemma \ref{a42} to bound the $\log$, which then can be removed from the integral.
    \[
    (\star)\le\eps^2\cdot\frac{(n+3)(n+2)(n+1)}{4\pi}\log\pare{\frac{4e\magn{M'}}\eps}\int_0^{2\pi}\frac1{w_1(\theta)w_2(\theta)}\wrt\theta.
    \]
    Then Lemma \ref{a44} bounds the integral resulting in
    \spliteq{}{
    (\star)
      &\le\eps^2\cdot\frac{(n+3)(n+2)(n+1)}{4\pi}\log\pare{\frac{4e\magn{M'}}\eps}\cdot\frac{4\pi+16\log\pare{{12}^{1/4}\cdot\magn{M'}/\eps}}{\abs{\phi(M')}^{1/2}}
    \\&\le\eps^2\log^2\pare{4e\magn{M'}/\eps}\cdot\frac4\pi\cdot{(n+3)(n+2)(n+1)}\frac1{\abs{\phi(M')}^{1/2}}
    }
    Then Lemma \ref{a47} provides a lower bound on $\phi(M')$, resulting in
\[(\star)\le\eps^2\log^2\pare{4e\magn{M'}/\eps}\cdot\frac4\pi\cdot{(n+3)(n+2)(n+1)}\frac4{\sigma_9(M')\inR(W(M'))}.\]
Finally, $\sigma_9(M')\ge\sigma_9(M)$ and $\inR(W(M'))\ge\inR(W(M))$ are both deterministic facts and reduce the bound to be in terms of $M$.
\end{proof}

\section{Least singular value}\label{a13}
\newcommand{\xlogx}[1]{#1\log(e/#1)}
The goal of this section is to use the improved small-ball estimates for the numerical measure from Section \ref{a10} along with the reduction from Section \ref{a6} to obtain a stronger tail bound on $\smin(Q^*AQ)$ than what's provided in Section \ref{a8}.
Recall our reduction in Lemma \ref{a7} derives tail bounds on $\smin(Q^*AQ)$ from the anti-concentration of the numerical measure of a random Schur complement $M=(A/Q'),Q'\sim\semiu(n,\ell-1)$. Proposition \ref{a12} shows that for each fixed $M$, this anti-concentration can be expressed in terms of three quantities of $M$:
\spliteq{\label{a48}}{
\magn M\qand\sigma_9(M)\qand\inR(W(M)).}
These quantities must be estimated when $M$ is a random Schur complement of $A$ in order for Proposition \ref{a12} to contribute to a tail bound on $\sigma(Q^*AQ)$.
Note that
\[\magn M=\magn{A-AQ'(Q'^*AQ')^{-1}Q'^*A}\]
can be large when $\magn{(Q'^*AQ)^{-1}}^{-1}=\smin(Q'^*AQ')$ is small. But this is exactly the tail bound we're seeking control over!
Fortunately, we already derived a tail bound for this: Proposition \ref{a25}. This is a weaker tail bound, but because $\magn M$ appears under a log, it turns out to be good enough. Control over $\sigma_9(M)$ will follow from Eckart-Young since $M-A$ will be low rank. The difficult term to control is $\inR(W(M))$.

To control $\inR(W(M))$, we first establish control over the area of the numerical range of random compressions. For concision, we introduce the following notation: for set $\Omega\subset\C$ denote
\[\Omega\bowtie 0=\text{convex hull of }\Omega\text{ and the origin}.\]

\begin{lemma}[Random compression area]
\label{a49}
Fix any $B\in\C^{m\times m}$ and let $T\sim\semiu(m,2k)$.
Then for every $\theta\in(0,1)$,
\[
\Pr\pare{
\area(W\pare{T^*BT}\bowtie0)\le\frac{\area(W(B)\bowtie0)}{4\pi m^2}\cdot\theta}
\le\xlogx{\theta^k}.
\]
\end{lemma}
\begin{proof}
Note that $\col(T)$ is generated by $2k$ \textit{independent} samples from the Haar distribution on $\semiu(m,1)$. Thus, if $x_1,\ldots,x_k,y_1,\ldots y_k$ are independent samples from the numerical measure of $B$, then
\spliteq{}{
&\conv(0,W\pare{U^*BU})\supset\conv\set{x_1,\ldots,x_k,y_1,\ldots y_k,0}\supset\bigcup_{j,j'\in[k]}\conv\set{x_j,y_{j'},0}
\\
\implies
&\area\conv(0,W\pare{U^*BU})\ge\max_{j,j'\in[k]}\area\conv\set{x_j,y_{j'},0}.}
Note that the areas appearning on the right hand side if we replace $x_j,y_{j'}$ with samples from the numerical measure of $B'=B\oplus\bmat0$ instead. Note $W(B')=W(B)\bowtie0$.
Let $r_-,r_+$ be the inner and outer radii of $W(B')$ respectively. Recall that $W(H(\eitc B'))$ is the projection of $W(B')$ onto the line $\eit\R$ rotated back to the real axis, which is an interval of length $w_1(H(\eitc B'))$. Thus, there exists $\theta_+$ with $w_1(H(e^{-i\theta_+} B'))\ge r_+$ and $w_1(H(\eitc B'))\ge2r_-$ for all $\theta$.

We follow a similar strategy as \eqref{a26} and use \(\abs{v^*B'v}\ge\abs{v^*H(e^{-i\theta_+} B')v}\) along with
with Fact \ref{a22} and Lemma \ref{a24} to obtain for $q\sim\semiu(n,1),$
\spliteq{\label{a50}}{
\Pr\pare{ \abs{v^*B'v}\le s }
\le\Pr\pare{ \abs{v^*H(e^{-i\theta_+} B')v}\le s }
\le\frac{m}{w_1\pare{H(e^{-i\theta_+} B')}}\cdot2s
\le\frac{m}{r_+}\cdot2s.}
Let $x$ denote the $x_j$ with largest absolute value and set $\theta=\arg(x)$. 
Define the set
\[
\Omega_{x,\eps}
=\set{z\in W(B'):\area \conv\set{z,x,0}\le\eps}
=\set{z\in W(B'):\Re(\eitc z)\le\frac\eps{\abs x}}.\]
Given a fixed value of $x$, the probability any $y_j$ lands in $\Omega_{x,\eps}$ is simply
\spliteq{\label{a51}}{
\Pr\pare{v^*B'v\in\Omega_{x,\eps}\,|\,x}
=\Pr\pare{ \abs{v^*H(\eitc B')v}\le\frac{\eps}{\abs x}\,\bigg|\,x}
\le\frac{m}{w_1(H(\eitc B'))}\cdot\frac{2\eps}{\abs x}
\le\frac{m}{r_-}\cdot\frac{\eps}{\abs x}.
}
The resulting bound follows by applying iterated expectation. First we control the inner probability with \eqref{a51}.
\spliteq{}{
(\star)=\Pr\pare{\area\conv\set{x_{j'},y_{j},0}\le\eps,\,\forall j,j'\in[k]}
  &\le\E\Pr\pare{\area\conv\set{x,y_{j},0}\le\eps,\,\forall j\in[k]\,|\,x}
\\&=  \E\Pr\pare{\area\conv\set{x,y_{j},0}\le\eps\,|\,x}^k
\\&\le\E\min\pare{1,\pare{\frac{m}{r_-}\cdot\frac{\eps}{\abs x}}^k}
}
Now use the variational formula for expectation along with the definition of $x$ as the largest absolute $x_j$.
\spliteq{}{
(\star)
\le\int_0^1\Pr\pare{ \pare{\frac{m}{r_-}\cdot\frac{\eps}{\abs x}}^k\ge t }\wrt t
=  \int_0^1\Pr\pare{ \frac{\eps m}{t^{1/k}r_-}\ge \abs x }\wrt t
=  \int_0^1\Pr\pare{ \frac{\eps m}{t^{1/k}r_-}\ge \abs{x_j} }^k\wrt t.}
This is controlled by \eqref{a50},
\[
(\star)\le\int_0^1\min\pare{ 1,\pare{\frac{2\eps m^2}{r_-r_+}}^k \cdot\frac1t }\wrt t.
\]
Set $a=\pare{\frac{2\eps m^2}{r_+r_-}}^k$. When $a\le 1$ the integral is
\[
\int_0^1\min(1,a/t)\wrt t
=\int_0^a1\wrt t+\int_a^1a/t\wrt t
=a+a\log(1/a)=\xlogx a.\]
Finally, use $r_-r_+\ge\frac1{2\pi}\area W(B')$ to bound $a$.
\end{proof}
We now turn this into a lower tail bound on \(\inR(W(M)).\)
\begin{lemma}[Control on inner radius]\label{a52}
Fix $A\in\C^{n\times n}$ and set $M=(A/Q')$ for $Q'\sim\semiu(n,\ell-1)$.
Then for every $\ell'>\ell$,
\[
\Pr\pare{
\inR(W((A/Q')))
\ge
\frac\theta{(70\ell'n)^{2.5}}
\cdot\frac{\sigma_{\ell'}(A)^2}{\sigma_1(A)^2}\cdot\inR(W(A)\bowtie0)}
\ge1-\xlogx{\theta^{\frac{\ell'-\ell}2}}.
\]
\end{lemma}
\begin{proof}
We prove the result for nonsingular $A$, which extends to all $A$ by continuity.
The first step is to relate the numerical range of $(A/Q')$ to the numerical range of a compression of the inverse.
Fix a $U\in\semiu(n,\ell')$ (we'll say how to pick $U$ later) and set $V\in\semiu(n,\ell')$ so that $\col(V)=\col(AU)$.

Note that \((A/Q')x=Ax\) for all $x\in\col(A^*Q')^\perp$. So
\spliteq{}{
W((A/Q'))
   \supset\set{\frac{x^*Ax}{\magn x^2}:x\in\col(A^*Q')^\perp}
  &\supset\set{\frac{x^*Ax}{\magn x^2}:x\in\col(A^*Q')^\perp\cap\col(U)}
\\&=      \set{\frac{y^*A^{-*}y}{\magn{A^{-1}y}^{2}}:y\in\col(Q')^\perp\cap\col(AU)}.
}
Consider the subspace $\col(Q')^\perp\cap\col(AU)=\col(Q')^\perp\cap\col(V)$. Since $Q'$ is Haar distributed, $\col(Q')^\perp\cap\col(V)$ is a Haar distributed subspace of $\col(V)$. Thus, it can be expressed as $\col(Q')^\perp\cap\col(V)=\col(VT)$ where $T\sim\semiu(\ell',\ell'-\ell)$. Since $W((A/Q'))$ is convex and contains the origin, we subsequently have
\spliteq{}{
W((A/Q'))
\supset\set{\frac{y^*A^{-*}y}{\magn{A^{-1}y}^{2}}:y\in\col(VT)}\bowtie0
&\supset\frac1{\sup_{\substack{y\in\col(V)\\\magn y=1}}\magn{A^{-1}y}^2}\set{\frac{y^*A^{-*}y}{\magn y^2}:y\in\col(VT)}\bowtie0
\\&=\smin(AU)^2\cdot W(T^*V^*A^{-*}VT)\bowtie0.}
$T^*V^*A^{-*}VT$ is a Haar random compression of $V^*A^{-1}V$, so the area of the right hand side is controlled by Lemma \ref{a49}. Specifically, we have
\[\area(W(T^*V^*A^{-*}VT)\bowtie0)\ge\frac{\theta}{4\pi(\ell')^2}\cdot\area(W(V^*A^{-*}V)\bowtie0)\]
with probability at least $1-\xlogx{\theta^{\frac{\ell'-\ell}2}}$. Condition on this event.
Now examine the numerical range on the right hand side,
\[W(V^*A^{-*}V)=\set{\frac{y^*A^{-*}y}{\magn y^2}:y\in\col(V)}=\set{\frac{x^*Ax}{\magn{Ax}^2}:x\in\col(U)}.\]
Similar to before, note $W(V^*A^{-1*}V)\bowtie0$ is convex and contains the origin so we can pull out the $\magn{Ax}$ factor,
\spliteq{}{
W(V^*A^{-*}V)\bowtie0
  =\set{\frac{x^*Ax}{\magn{Ax}^2}:x\in\col(U)}\bowtie0
  &\supset\frac1{\sup_{\substack{x\in\col(U)\\\magn x=1}}\magn{Ax}}\cdot W(U^*AU)\bowtie0
\\&=\sigma_1(AU)^{-2}\cdot W(U^*AU)\bowtie0.
}
The consequence is
\[W((A/Q'))\supset\frac{\theta}{4\pi(\ell')^2}\cdot\frac{\smin(AU)^2}{\sigma_1(AU)^2}\cdot W(U^*AU)\bowtie0.\]
Consider for the moment sampling $U\sim\semiu(n,\ell-1)$. Then by Lemma \ref{a49},
\[
\area(W(U^*AU)\bowtie0)\ge\frac{0.1}{4\pi n^2}\cdot\area W(A)\bowtie0
\]
with probability at least $2/3$. By rotational invariance, the quantity $\smin(AU)$ has the same distribution if we replace $A$ by the diagonal matrix of it's singular values via the SVD. Then applying Lemma \ref{a53} (Proposition C.3 in \cite{b2}) gives
\[
\smin(AU)\ge\sigma_{\ell'}(A)\sigma_{\ell'}(U)
\ge\frac{\sigma_{\ell'}(A)}{2\sqrt{\ell'(n-\ell')}}\]
with probability at least 3/4. By the union bound, there is a positive probability both these events occur. Fix $U$ such that this happens.
Note $W(U^*AU)\bowtie0\subset\ball(0,\magn A)$ so the in-radius of $W(U^*AU)\bowtie0$ is at least
\[
\inR(W(U^*AU)\bowtie0)
\ge\frac{\area(W(U^*AU)\bowtie0)}{2\pi\magn A}
\ge\frac{0.1}{4\pi n^2}\,\frac{\area(W(A)\bowtie0)}{2\pi\magn A}
\ge\frac{\inR(W(A)\bowtie0)}{160\pi^2n^2}
.\]
Consequently, the in-radius of $W((A/Q'))$ is at least
\spliteq{}{
\inR(W((A/Q')))
  &\ge
\frac{\theta}{4\pi(\ell')^2}\cdot\frac{\smin(AU)^2}{\sigma_1(AU)^2}
\cdot
\inR(W(U^*AU)\bowtie0)
\\&\ge
\frac{\theta}{4\pi(\ell')^2}\cdot\frac1{2(\ell')^{1/2}(n-\ell')^{1/2}}\frac{\sigma_{\ell'}(A)^2}{\sigma_1(A)^2}
\cdot
\frac{\inR(W(A)\bowtie0)}{160\pi^2n^2}.
\\&\ge
\frac\theta{40000}\cdot\frac1{(\ell')^{2.5}n^{2.5}}
\cdot\frac{\sigma_{\ell'}(A)^2}{\sigma_1(A)^2}\cdot\inR(W(A)\bowtie0)
}
with probability at least $1-\xlogx{\theta^{\frac{\ell'-\ell}2}}$.
\end{proof}

We now have control over all three of the required quantities \eqref{a48}. Therefore, Proposition \ref{a12} implies a tail bound on $\smin(Q^*AQ)$.
Our proof of this fact uses the notion of stochastic domination. We say that a random variable $X$ dominates a random variable $Y$ if $\Pr(X\le\theta_0)\le\Pr(Y\le\theta_0)$ for all $\theta_0$. This implies a coupling between $X$ and $Y$ such that $Y\le X$ almost surely.

\begin{proposition}\label{a15}
    Fix any $A\in\C^{n\times n},\eps\in(0,\magn M/2)$. Pick any $\ell\le n-8$ and sample $Q\in\semiu(n,\ell)$. Then
\[\Pr\pare{\smin(Q^*AQ)\le\eps}\le\eps^2\log^2\pare{\frac{c'_{n,\ell}}\eps\cdot\frac{2\magn A^2}{\sigma_{\ell-1}(z-A)}}\cdot\frac{\magn A^2}{\sigma_{\ell+8}(A)\sigma_{\ell+3}(A)^2}\cdot\frac{c_{n,\ell}}{\inR(A)}.\]
where $\constepst=O(n^{5.5}\ell^{2.5})$ and $\constlog=O(\ell^3n)$ are absolute polynomial factors.
\end{proposition}
\begin{proof}
Let $M=(A/Q')$ for $Q'\sim\semiu(n,\ell)$.
The reduction in Lemma \ref{a7} gives
\[
\Pr\pare{\smin(Q^*AQ)\le\eps}
\le3\ell^2\E\Pr\pare{\abs{q^*Mq}\le2\ell\eps\,|\,M}.
\]
Then apply Proposition \ref{a12} to obtain
\[
(\star):=\Pr\pare{\abs{q^*Mq}\le2\ell\eps\,|\,M}
\le
\eps^2\log^2(4e\magn M/\eps)\cdot
\frac
{5.1\,(n+3)^3}
{\sigma_9(M)\inR(W(M))}.\]
We may apply the deterministic bound $\sigma_{9}(M)\ge\sigma_{\ell+8}(A)$, but the remaining factors are random variables with possibly unbounded support.
A rephrasing of Lemma \ref{a52} is that the ratio
\[
\inR(W(M))\bigg/\inR(W(A))\cdot\frac{\sigma_{\ell'}(z-A)^2}{\sigma_1(A)^2}\cdot\frac1{(70\ell'n)^{2.5}}
\]
stochastically dominates the random variable $\theta$ that satisfies $\Pr(\theta\le\theta_0)=\xlogx{\theta_0^{\frac{\ell'-\ell}2}}$.
We similarly state a rephrasing of Proposition \ref{a25}. The ratio
\[
\smin(Q'^*AQ')
\bigg/
\frac{\sigma_{\ell-1}(A)}{24(\ell-1)^3(n-1)}
\]
stochastically dominates the random variable $\theta'$ that satisfies $\Pr(\theta'\le\theta_0)=\theta_0$.
Note
\spliteq{}{
\magn{M}
=\magn{A-AQ(Q'^*AQ')^{-1}Q^*(A)}
\le\magn{A}+\frac{\magn{A}^2}{\smin(Q^*AQ)}
\le\frac{2\magn{A}^2}{\smin(Q'^*AQ')}.}
Given the appropriate coupling of $Q'$ and $\theta,\theta'$, we have the upper bound with probability 1,
\spliteq{}{
(\star)
&\le
\eps^2\log^2\pare{\frac{4e}\eps\cdot\frac{2\magn{A}^2}{\smin(Q'^*AQ)}}\cdot
\frac
{5.1\,(n+3)^3}
{\sigma_{8+\ell}(A)\inR(M)}
\\&\le
\eps^2\log^2\pare{\frac{4e}\eps\cdot\frac{2\magn{A}^2\cdot24(\ell-1)^3(n-1)}{\sigma_{\ell-1}(A)}\cdot\frac1{\theta'}}\cdot
\frac
{5.1\,(n+3)^3}
{\sigma_{8+\ell}(A)\inR(A)}\cdot\frac{\sigma_1(A)^2}{\sigma_{\ell'}(A)^2}\cdot\frac{(70\ell'n)^{2.5}}{\theta}
}
Pick $\ell'=\ell+3$. Then taking expectations gives
\spliteq{}{
\E(\star)
&\le10
\eps^2\log^2\pare{\frac{4e}\eps\cdot\frac{2\magn{A}^2\cdot24(\ell-1)^3(n-1)}{\sigma_{\ell-1}(A)}}\cdot
\frac
{5.1\,(n+3)^3}
{\sigma_{8+\ell}(A)\inR(A)}\cdot\frac{\sigma_1(A)^2}{\sigma_{\ell'}(A)^2}\cdot{(70\ell'n)^{2.5}}.
}
\end{proof}

\section{Pseudospectral area}\label{a16}
The goal of this section is to convert the tail bounds on $\smin(\cdot)$ from Lemma \ref{a54} and Proposition \ref{a15} into bounds on the expected pseudospectral area.
In this section, fix any $A\in\C^{n\times n}$ and $\eps\in(0,\magn A)$.
The results are stated in terms of $s_k,r,$ and $R$, where $s_k$ is smallest $k$th singular value among shifts:
\[z_k=\argmin_{z\in\C}\sigma_k(z-A)\qand s_k=\sigma_k(z_k-A),\]
and $R$ and $r$ are the diameter and inner radius of
\[\Omega=W(A)+\ball(0,\eps)\]
respectively. The proofs will partition $\Omega$ into the contours
\[\gamma_{k,x}=\set{z\in\Omega:\abs{z-z_k}=x}\]
which by construction have length
\spliteq{\label{a55}}{\len(\gamma_{k,x})\le2\pi\min\pare{x,r}}
The main idea is to use the following fact to convert tail bounds on $\smin(\cdot)$ into the expected area of $\Lambda_\eps(\cdot)$.
\begin{fact}\label{a56}
\[\E\area\Lambda_\eps(Q^*AQ)=\int_{z\in\Omega}\Pr\pare{\smin(Q^*(z-A)Q)\le\eps}.\]
\end{fact}
\begin{lemma}\label{a57}
For $k\le(n+1)/2$,\[\sigma_k(z-A)\ge\max\pare{s_k,\frac{\abs{z-z_k}}2}.\]
\end{lemma}
\begin{proof}
$s_k\le\sigma_k(z-A)$ follows by definition of $s_k$.
\[\abs{z-z_k}
=\sigma_n((zI-A) - (z_kI-A))
\le\sigma_{n-k+1}(z_k-A)+\sigma_k(z-A)
\]
If $k\le n-k+1$, then
\[\sigma_{n-k+1}(z_k-A)\le\sigma_k(z_k-A)\le\sigma_k(z-A).\]
\end{proof}

Our first bound uses the first order tail bound from Proposition \ref{a25} to derive a pseudo-spectral area bound on the order of $r\eps$.
\begin{lemma}\label{a54}
\[\E\area\Lambda_\eps(Q^*AQ)=2\pi\constepso\log\pare{\frac{eR}{\max\pare{\constepso\eps,r,s_\ell}}}\cdot r\eps.\]
\end{lemma}
\begin{proof}
Apply Proposition \ref{a25} to the right hand side of the inequality from Fact \ref{a56}.
\spliteq{}{
\E\area\Lambda_\eps(Q^*AQ)
\le\int_{z\in\Omega}\Pr\pare{\smin(Q^*(z-A)Q)\le\eps}
\le\int_{z\in\Omega}\min\pare{1,\frac{\constepso\eps}{\sigma_\ell(z-A)}}
}
Using the lower bound on $\sigma_\ell(z-A)$ from Lemma \ref{a57} gives
\[
\E\area\Lambda_\eps(Q^*AQ)
\le\int_{z\in\Omega}\min\pare{1,\frac{\constepso\eps}{\abs{z-z_{\ell}}},\frac{\constepso\eps}{s_\ell}}.\]
We can partition the region of integration into the union of $\gamma_{\ell, x}$ and apply Lemma \ref{a55},
\spliteq{}{
\E\area\Lambda_\eps(Q^*AQ)
&\le\int_0^R\len(\gamma_{\ell,x})\min\pare{1,\frac{\constepso\eps}{x},\frac{\constepso\eps}{s_\ell}}\wrt x
\\&\le2\pi\int_0^R\min(x,r)\min\pare{1,\frac{\constepso\eps}{x},\frac{\constepso\eps}{s_\ell}}\wrt x
\\&=2\pi\int_0^R\min\pare{\frac{a_{-1}}x,a_0,a_1x}\wrt x
.}
where
\spliteq{}{
a_{-1}=\constepso r\cdot\eps
\qand
a_0=\min\pare{r,\frac{a_{-1}}r,\frac{a_{-1}}{s_\ell}}
\qand
a_1=\min\pare{1,\frac{a_{-1}}{r\cdot s_\ell}}.}
First assuming that $a_0\le\pare{a_1a_{-1}}^{1/2}$, the integral becomes
\spliteq{}{
 \int_0^{a_0/a_1}a_1x\wrt x+\int_{a_0/a_1}^{a_{-1}/a_0}a_0\wrt x+\int_{a_{-1}/a_0}^R\frac{a_{-1}}x\wrt x
  &=\frac{a_0^2}{2a_1}+\pare{a_{-1}-\frac{a_0^2}{a_1}}+a_{-1}\log(Ra_0/a_{-1})
\\&\le a_{-1}\log(eRa_0/a_{-1}).}
If $a_0\ge\pare{a_1a_{-1}}^{1/2}$, the integral becomes
\spliteq{}{
 \int_0^{\pare{a_{-1}/a_1}^{1/2}}a_1x\wrt x+\int_{\pare{a_{-1}/a_1}^{1/2}}^R\frac{a_{-1}}x\wrt x
  &=\frac{a_{-1}}2+a_{-1}\log(R(a_1a_{-1})^{1/2}/a_{-1})
\\&\le a_{-1}\log(e^{1/2}R(a_1a_{-1})^{1/2}/a_{-1})
.}
This makes the overall bound
\spliteq{}{
\E\area\Lambda_\eps(Q^*AQ)
&\le2\pi a_{-1}\log\pare{\frac{eR\min\pare{(a_1a_{-1})^{1/2}, a_0}}{a_{-1}}}.
\\&=2\pi\constepso r\eps\log\pare{\frac{eR}{\max\pare{\constepso\eps,r,s_\ell}}}.
}
\end{proof}

\begin{theorem}\label{a2}
The following four quantities are upper bounds on $\E\area\Lambda_\eps(Q^*AQ)$, where $Q\sim\semiu(n,\ell)$ for $\ell\le n/2-7.5$
\begin{enumerate}
    \item \[{4\pi\constepst}\log^2\pare{\frac{\constlog 2eR^2}{\eps s_{\ell+8}}}\cdot\frac{R^2}{s_{\ell+8}^2}\cdot\eps^2\]
    \item \[{4\pi\constepst}\log^2\pare{\frac{\constlog 2eR^2}{\eps s_{\ell+8}}}\cdot\frac{R^2}{s_{\ell+8}r}\cdot\eps^2\]
    \item \[4\pi\constepst^{1/3}\log^2\pare{\frac{\constlog 2eR^{4/3}r^{1/3}}{\constepst^{1/3}\eps^{5/3}}}\cdot(Rr)^{2/3}\cdot\eps^{2/3}\]
    \item \[4\pi\constepst^{2/3}\log^2\pare{\frac{\constlog 2eR^{4/3}r^{1/3}}{\constepst^{1/3}\eps^{5/3} }}\cdot\frac{R^{4/3}}{r^{2/3}}\cdot\eps^{4/3}\]
    \item \[25(\constepst\constepso)^{2/5}\log(nR/\eps)\cdot R^{4/5}\cdot\eps^{6/5}.\]
\end{enumerate}
\end{theorem}
\begin{proof}
From Proposition \ref{a15}, we have
\[\Pr\pare{\smin(Q^*(z-A)Q)\le\eps}
\le\min\pare{1,
\eps^2\log^2\pare{\frac{\constlog}\eps\cdot\frac{2\magn{z-A}^2}{\sigma_{\ell-1}(z-A)}}\cdot\frac{\magn{z-A}^2}{\sigma_{\ell+8}(z-A)\sigma_{\ell+3}(z-A)^2}\cdot\frac{\constepst}{\inR(A)}}.
\]
We may two approximations to simplify this bound. The first is $\sigma_{\ell-1}(z-A)\ge\sigma_{\ell+3}(z-A)\ge\sigma_{\ell+8}(z-A)\ge\max(s_{\ell+8},\abs{z-z_{\ell+8}})$  by Lemma \ref{a57}. The second is $\magn{z-A}\le R$ for $z\in\Omega$. Plugging in these bounds for $z\in\gamma_{\ell+8,x}$ gives
\[
Pr\pare{\smin(Q^*(z-A)Q)\le\eps}
\le f(x):=\min\pare{1,\frac a{\max(s_{\ell+8},x)^3}\log^2\pare{\frac b{\max(s_{\ell+8},x)}}}
\]
where
\[
a=\frac{\eps^2\constepst R^2}{r}
\qand b=\frac{\constlog\cdot 2R^2}\eps.\]
Then by Fact \ref{a56} and Lemma \ref{a55}, we have
\spliteq{}{
\E\area\Lambda_\eps(Q^*AQ)
  \le\int_0^R\len(\gamma_{\ell+8,x})f(x)\wrt x
  \le2\pi\int_0^R\min(x,r)f(x)\wrt x.
}
We split the integral into multiple parts. We let $t$ denote one split point, and will specify $t$ later. When $t\le r$, we break the integral into two parts.
\spliteq{}{
\int_0^Rxf(x)\wrt x
  &\le\int_0^txf(x)\wrt x+\int_t^\infty xf(x)\wrt x
\\&\le\frac{t^2f(0)}2+\frac{a\log^2(b/t)}{t}
.}
When $t\ge r$, we break the integral into three parts.
\spliteq{}{
\int_0^R\min(x,r)f(r)\wrt r
&\le
 \int_0^rxf(0)\wrt x
+r\int_x^tf(0)\wrt x
+r\int_t^Rf(x)\wrt x
\\&\le
 \frac{r^2f(0)}2
+r(t-r)f(0)
+r\frac a2\log^2(b/t)\pare{\frac1{t^2}-\frac1{R^2}}
\\&\le
rtf(0)+\frac{ra}{2t^2}\log^2(b/t)
.}
We therefore have the overall bound of
\[
\E\area\Lambda_\eps(Q^*AQ)
\le
2\pi\inf_{t>0}\min(r,t)\pare{tf(0)+\frac a{t^2}\log^2(b/t)}.\]
We set $t=\max(s_{\ell+8},a^{1/3})$. When $a^{1/3}\le s_{\ell+8}$, note
\[f(0)\le\frac a{s_{\ell+8}^3}\log^2(b/s_{\ell+8})=\frac a{t^3}\log^2(b/t)\]
so
\[\E\area\Lambda_\eps(Q^*AQ)\le4\pi\min(r,t)\frac a{t^2}\log^2(b/t).\]
When $a^{1/3}\ge s_{\ell+8}$, note \[f(0)\le1=\frac a{t^3}\] so
\[
\E\area\Lambda_\eps(Q^*AQ)
\le
2\pi\min(r,t)\frac a{t^2}\log^2(eb/t).\]
The combined bound is therefore
\spliteq{}{
\E\area\Lambda_\eps(Q^*AQ)
  &\le4\pi\min(r,t)\frac{a}{t^2}\log^2(eb/t)
\\&=  4\pi\min\pare{\frac{ra}{t^2},\frac{a}{t}}\log^2(eb/t)
\\&=  4\pi\min\pare{\frac{ra}{s_{\ell+8}^2},\frac{a}{s_{\ell+8}},\frac{ra}{a^{2/3}},\frac{a}{a^{1/3}}}\log^2(eb/t)
\\&=  4\pi\min\pare{
\frac{ra}{s_{\ell+8}^2},
\frac{ra}{rs_{\ell+8}},
r^{2/3}(ra)^{1/3},
r^{-2/3}(ra)^{2/3}
}\log^2(eb/t)
.}
Plugging in $ra=(R\eps)^2\constepst$ and taking each of the four arguments of the minimum gives the first four bounds in the theorem statement. For the fifth bound, take the $\frac25$-$\frac35$ weighted geometric mean of the fourth bound and the bound from Lemma \ref{a54} to eliminate the $r$ from the right hand side.
\end{proof}

\begin{repcoro}{a3} \textit{Given $A\in\C^{n\times n}$, set
\[
\beta=\begin{cases}
    6/5 & \text{if assuming (a)}\\
    4/3 & \text{if assuming (a) and (b)}\\
    2 & \text{if assuming (a) and (c)}
\end{cases}.
\]
Then
\spliteq{}{
\E\area(\Lambda_\eps(Q^*AQ))\le\poly(n)\log^2(1/\eps)\cdot\eps^\beta}
where $Q\sim\semiu(n,\ell)$ for $\ell\le n/2-8$.
}\end{repcoro}
\begin{proof}
    Under assumption (a), that $W(A)$ is contained in a disk of radius $\poly(n)$, we can set $R=\poly(n)+\eps$ in the fifth bound from Theorem \ref{a2}. Under the additional assumption (b), that $W(A)$ contains a disk of radius $1/\poly(n)$, we can set $r=1/\poly(n)+\eps$ in the fourth bound from Theorem \ref{a2}. Instead if we additionally assume (c), that $s_{\ell+8}\ge1/\poly(n)$, then we can apply the first bound from Theorem \ref{a2}.
\end{proof}

\bibliographystyle{alpha}
\bibliography{outbib}

\appendix

\section{Integral calculation for Lemma \ref{a44}}
\begin{lemma}\label{a45}
    Fix real $a,b,\eps,\theta_0$ with $a,\eps>0$. Then
    \[
    \int_0^\pi\frac1{\max\pare{\eps^2,\,\abs{b+a\cos(2\theta-\theta_0)}}}\wrt\theta
    \le\frac{4\pi+16\log\pare{\sqrt2\max\left(\frac{\sqrt[4]{\abs{a^2-b^2}}}{\eps},1\right)}}{\sqrt{\abs{a^2-b^2}}}.\]    
\end{lemma}
\begin{proof}
By periodicity, we may assume without loss of generality that $\theta_0=0$.
When $a<\abs b$, there is an antiderivative for $\frac1{b+a\cos(2\theta)}$ which gives an upper bound of
\[\frac\pi{\sqrt{\abs{a^2-b^2}}}.\]
So now assume $a\ge\abs b$.
Let $g(\theta)=b+a\cos(2\theta)$.
By symmetry, it suffices to integrate over $\theta\in[0,\pi/2]$.
We can express the integral explicitly as a Lebesgue integral,
\spliteq{}{
\int_0^{\pi/2}g(\theta)\wrt\theta
  &=\int_0^{\infty}\mu\pare{\set{\theta\in[0,\pi/2]:\max(\eps^2,\abs{g(\theta)})^{-1}\ge t}}\wrt t
\\&=\int_0^{1/\eps^2}\mu\pare{\set{\theta\in[0,\pi/2]:\abs{g(\theta)}\le t^{-1}}}\wrt t
\\&=\int_{\eps^2}^{\infty}x^{-2}\mu\pare{\set{\theta\in[0,\pi/2]:\abs{g(\theta)}\le x}}\wrt x.}
The set we need the measure of will simply be an interval, with end points either the intersection of $\abs{g(\theta)}$ and $t$ or the endpoints of the domain, $0,\pi/2$. In particular, the measure is simply the difference in the values of the endpoints. By rearranging $\pm g(\theta)=t$, we obtain an explicit expression for the integrand. Set
\[
f_1(x)=\begin{cases}
    x^{-2}\cos^{-1}\pare{-\frac{x+b}a} & 0\le x\le a-b
    \\
    x^{-2}\pi & a-b\le x
\end{cases},
\qand
f_2(x)=\begin{cases}
    x^{-2}\cos^{-1}\pare{\frac{x-b}a} & 0\le x\le a+b
    \\
    0 & a+b\le x
\end{cases}.
\]
Then the integrand is $f_1(x)-f_2(x)$. There are continuous anti-derivatives available. Let $x'=\min(x,a-b)$.
\[
\int f_1=F_1(x)=
\frac{\log\left(x'\right)-\log\left(\sqrt{a^{2}-b^{2}}\sqrt{a^{2}-\left(b+x'\right)^{2}}+a^{2}-b\left(b+x'\right)\right)}{\sqrt{a^{2}-b^{2}}}-\frac{\cos^{-1}\left(-\frac{b+x'}{a}\right)}{x}.
\]
Let $x''=\min(x,a+b)$.
\[
\int f_2=F_2(x)=
\frac{\log\left(\sqrt{a^{2}-b^{2}}\sqrt{a^{2}-\left(b-x''\right)^{2}}+a^{2}+b\left(x''-b\right)\right)-\log\left(x''\right)}{\sqrt{a^{2}-b^{2}}}-\frac{\cos^{-1}\left(\frac{x''-b}{a}\right)}{x}.
\]
Note $F_1(\infty)=-F_2(\infty)=\frac{\log(1/a)}{\sqrt{a^2-b^2}}$. So the bound we seek is $F_2(\eps^2)-F_1(\eps^2)$. We first deal with the $\cos^{-1}$ terms. Start by observing with some case work that
\[
\frac{\cos^{-1}\left(-\frac{b+x'}{a}\right)}{x}
-
\frac{\cos^{-1}\left(\frac{x''-b}{a}\right)}{x}
\le\min\pare{\frac\pi x,\,\frac{\cos^{-1}\pare{2\cdot\frac{\abs b}a-1}}{a-\abs b}}
\le\min\pare{\frac\pi x,\,\frac\pi{\sqrt{a^2-b^2}}}
\]
where the second inequality came from $\cos^{-1}(2z-1)\le\pi\sqrt{(1-z)/(1+z)}$ for $z\in[0,1]$. Now for the $\log$ terms. Examine just the numerators. Note they will only be increased by removing the $(b+x')^2$ and $(b-x'')^2$ terms. So we have
\spliteq{}{
&\pare{\log\pare{\sqrt{a^2-b^2}\cdot a+a^2+b(x''-b)}-\log(x'')}
-\pare{\log(x')-\log\pare{\sqrt{a^2-b^2}\cdot a +a^2-b(b+x')}}
\\=
&\log\pare{\sqrt{a^2-b^2}\cdot a+a^2+b(x''-b)}+\log\pare{\sqrt{a^2-b^2}\cdot a +a^2-b(b+x')}-\log(x'x'')
\\\le
&\log\pare{\sqrt{a^2-b^2}\cdot a+a^2+ba}+\log\pare{\sqrt{a^2-b^2}\cdot a +a^2-ba}-\log(x'x'')
\\=
&2\log(a)+\log\pare{\sqrt{a^2-b^2}+a+b}+\log\pare{\sqrt{a^2-b^2}+a-b}-\log(x'x'')
\\=
&2\log(a)+\log\pare{2\cdot\frac{a+b}{x''}}+\log\pare{2\cdot\frac{a-b}{x'}}.}
The $2\log(a)$ terms cancel with the contributions from $F_1(\infty)=-F_2(\infty)=\log(1/a)/\sqrt{a^2-b^2}$, so all that remains is to estimate
\spliteq{}{
\log\pare{2\cdot\frac{a+b}{x''}}+\log\pare{2\cdot\frac{a-b}{x'}}
&=
\log\pare{4\max\pare{\frac{a+b}{x},1}\max\pare{\frac{a-b}{x},1}}
\\&\le
\log\pare{4\max\left(\frac{{a^2-b^2}}{x},1\right)}.}
We therefore have an overall bound of
\[
F_1(\infty)-F_2(\infty)-F_1(\eps^2)+F_2(\eps^2)
\le\frac{\pi+4\log\pare{\sqrt2\max\left(\frac{\sqrt[4]{a^2-b^2}}{\eps},1\right)}}{\sqrt{a^2-b^2}}.
\]
We need to multiply by $4$ because we integrated over $[0,\pi/2]$ instead of $[0,2\pi]$.
\end{proof}

\begin{lemma}[Proposition C.3 in \cite{b2}]\label{a53}
    Let $X$ be the $r\times r$ corner of $U\sim\semiu(n,r)$. Then
    \[\Pr\pare{ \sigma_r(X)\ge\frac{\sqrt{r(n-r)}}\theta }\ge1-\theta^2.\]
\end{lemma}

\end{document}